\documentclass[12pt]{amsart}
\usepackage{graphicx}
\usepackage{amssymb}
\usepackage{verbatim}
\usepackage{array}
\usepackage{latexsym}
\usepackage{enumerate}
\usepackage{amsmath}
\usepackage{amsfonts}
\usepackage{amsthm}
\usepackage[english]{babel}

\newcommand{\NN}{\mathbb N}
\newcommand{\ZZ}{\mathbb Z}

\def\cA{\mathcal A}

\def\cE{\mathcal E}

\def\cH{\mathcal H}

\def\cO{\mathcal O}

\def\cX{\mathcal X}

\def\l{\ell}
\def\Aut{\mbox{\rm Aut}}
\def\K{\mathbb{K}}
\def\PG{{\rm{PG}}}

\def\det{\mbox{\rm det}}

\def\Aut{\mbox{\rm Aut}}

\def\fqq{{\mathbb F}_{q^2}}
\def\fqs{{\mathbb F}_{q^2}}
\def\fls{{\mathbb F}_{\ell^2}}
\def\fl{{\mathbb F}_{\ell}}

\newcommand{\PGL}{\mbox{\rm PGL}}

\newcommand{\PGU}{\mbox{\rm PGU}}
\newcommand{\SU}{\mbox{\rm SU}}
\newcommand{\U}{\mbox{\rm U}}

\newcommand{\GL}{\mbox{\rm GL}}
\newcommand{\aut}{\mbox{\rm Aut}}

\newcommand{\ha}{{\textstyle\frac{1}{2}}}

\begin{document}

\title{Quotient curves of the GK curve}




\author{Stefania Fanali   and
        Massimo Giulietti 
}


\newtheorem{theorem}{Theorem}[section]
\newtheorem{proposition}[theorem]{Proposition}
\newtheorem{lemma}[theorem]{Lemma}
\newtheorem{corollary}[theorem]{Corollary}
\newtheorem{scholium}[theorem]{Scholium}
\newtheorem{remark}[theorem]{Remark}
\newtheorem{observation}[theorem]{Observation}
\newtheorem{assertion}[theorem]{Assertion}
\newtheorem{result}[theorem]{Result}
{\theoremstyle{definition}
\newtheorem*{definition*}{Definition}
\newtheorem{example}[theorem]{Example}
\newtheorem{rem}[theorem]{Remark}
\newtheorem*{proposition*}{Proposition}
\newtheorem*{corollary*}{Corollary}
\newtheorem*{lemma*}{Lemma}

\begin{abstract}
For every $q=\ell^3$
with $\ell$ a prime power greater than $2$, the GK curve $\cX$ is an $\fqq$-maximal curve  that is not
$\fqq$-covered by any $\fqq$-maximal Deligne-Lusztig curve.
Interestingly, $\cX$ has a very
large $\fqq$-automorphism group with respect to its genus.
In this paper we compute the genera of a large variety of curves that are Galois-covered by the GK curve, thus providing several new values in the spectrum of genera of $\fqs$-maximal curves.

\end{abstract}

\maketitle

\section{Introduction}\label{s1}

Let $\fqs$ be a finite field  with $q^2$ elements where $q$ is a
power of a prime $p$. An $\fqq$-rational curve, that is a
projective, geometrically
absolutely irreducible, non-singular algebraic curve defined over
$\fqq$, is called $\fqq$-maximal if the number of its
$\fqq$-rational points attains the  Hasse-Weil upper bound
$$q^2+1+2gq,$$
where $g$ is the genus of the curve. Maximal curves have
interesting properties and have also been investigated for their
applications in Coding theory. Surveys on maximal curves are found
in
\cite{garcia2001sur,garcia2002,garcia-stichtenoth1995ieee,geer2001sur,geer2001nato}
and \cite[Chapter 10]{hirschfeld-korchmaros-torres2008}; see also
\cite{fgt,ft,garcia-stichtenoth2007book,r-sti,sti-x}.

One of the most important problems on maximal curves is the determination of the possible genera of maximal curves over $\fqs$, see e.g. \cite{garcia2001sur}.
For a given $q$, the highest value of $g$ for which an $\fqs$-maximal curve of genus $g$ exists is $q(q-1)/2$ \cite{ihara}, and equality holds if and only if 
the curve is the Hermitian curve with equation $X^{q+1}=Y^q+Y$, up to $\fqs$-birational equivalence \cite{r-sti}. 

By a result of Serre, see \cite[Prop. 6]{lachaud1987}, any $\fqs$-rational curve which is $\fqs$-covered by an $\fqs$-maximal curve is also $\fqs$-maximal.
This has made it possible to obtain several genera of
$\fqs$-maximal curves by applying the Riemann-Hurwitz formula, especially from the Hermitian curve, see   \cite{abdon-quoos2004,abdon-torres1999, abdon-garcia2004, abdon-torres2005, cossidente-korchmaros-torres1999, cossidente-korchmaros-torres2000, garcia-kawakita-miura2006, garciaozbudak2007, garcia-stichtenoth-xing2001, garcia-stichtenoth1999bis,  giulietti-hirschfeld-korchmaros-torres2006, giulietti-hirschfeld-korchmaros-torres2006bis}. Others have been obtained from the
DLS and DLR curves, see  \cite{giulietti-korchmaros-torres2004,cakcak-ozbudak2004,cakcak-ozbudak2005,pasticci2006}.

The problem of the existence of $\fqs$-maximal curves other than  $\fqs$-subcovers of the Hermitian curve, the DLS curve, and the DLR curve  was solved in \cite{segovia1}, where for every $q=\ell^3$ with $\ell=p^r>2$, $p$ prime, an $\fqs$-maximal curve $\cX$ that is not $\fqs$-covered by any $\fqs$-maximal Deligne-Lusztig curve was described. Throughout the paper we will refer to $\cX$ as to the GK curve. It should be noted that the construction in \cite{segovia1} has been recently generalized in \cite{ggs}; it is still an open problem to determine whether these generalizations of the GK curve are $\fqs$-subcovers of a Deligne-Lusztig curve or not.

One of the most interesting features of the GK curve $\cX$ is its very large automorphism group with respect to its genus.
In this paper we consider quotient curves $\cX/L$ under the action of a large variety of subgroups $L$ of $\aut(\cX)$. By applying the Riemann-Hurwitz formula to the covering $\cX\to \cX/L$ a large number of genera of maximal curves is obtained, see Theorems \ref{thmA1}, \ref{thmA2}, \ref{thmB1}, \ref{thmB2}, \ref{thmB3},
\ref{thmC1}, \ref{thmC2}, \ref{thmC3}, \ref{thmC4}, \ref{thmC5}, \ref{7punto1}, \ref{7punto2}, \ref{7punto3}.
It should be noted that when $L$ is tame and contains the centrum $\Lambda$ of $\Aut(\cX)$, then the quotient curve $\cX/L$ has the same genus of the quotient curve of the Hermitian curve $\cH$ over $\fls$ under the action of the factor group $L/\Lambda$, see Corollary \ref{diellecor2}. Apart from these cases, formulas for genera of quotient curves $\cX/L$ appear to provide new values in the spectrum of genera of $\fqs$-maximal curve, cf. Section \ref{elle5}. One of our main tools for the investigation of the tame case  is a relationship between the genus of $\cX/L$ and that of the quotient curve of $\cH$ with respect to the factor group $L/(L\cap \Lambda)$, see Section \ref{section3}.

\section{The GK curve and its automorphism group}
Throughout this paper, $p$ is a prime, $\ell=p^h$ and $q=\ell^3$ with
$h\geq 1$, $\ell>2$. 
Furthermore, $\K$ denotes the algebraic closure of
$\fqq$.

Let
\begin{equation}
\label{segoviaeq} h(X)=\sum_{i=0}^n (-1)^{i+1}X^{i(n-1)}.
\end{equation}

In the three--dimensional projective space $\PG(3,q^2)$ over
$\fqq$, consider the algebraic curve $\cX$ defined to be the
complete intersection of the surface with affine equation
\begin{equation}
\label{segovia2} Z^{\ell^2-\ell+1}=Xh(Y),
\end{equation}
and the Hermitian cone with affine equation
\begin{equation}
\label{segovia4} Y^\ell+Y=X^{\ell+1}.
\end{equation}
Note that $\cX$ is defined over $\fqq$ but it is viewed as a curve
over $\K$. Moreover, $\cX$ has degree $\ell^3+1$ and possesses a
unique infinite point, namely the infinite point $P_\infty$ of the
$Y$-axis.

\begin{theorem}[\cite{segovia1}]
\label{segoviath} $\cX$ is an $\fqq$-maximal curve with genus
  $g=\ha\,(\ell^3+1)(\ell^2-2)+1$.
\end{theorem}
 For notation,
terminology and basic results on automorphism groups of curves,
we refer to \cite[Chapter 11]{hirschfeld-korchmaros-torres2008}.

For every $u\in \K$, with $u\neq 0$, consider the collineation
$\alpha_u$ of $\PG(3,\fqq)$ defined by
\begin{equation}
\label{lam} \alpha_{u}:\,\,(X,Y,Z,T)\to (uX,uY,Z,uT).
\end{equation}
For $u^{\ell^2-\ell+1}=1$, $\alpha_u$ defines an $\fqq$-automorphism of $\cX$.
For $u\neq 1$, the fixed points of $\alpha_{u}$ are exactly the points
of the plane $\pi_0$ with equation $Z=0$. Since $\pi_0$ contains
no tangent to $\cX$, the number of fixed points of $\alpha_{u}$ with
$u\neq 1$ is independent from $u$ and equal to $\ell^3+1$.
Let $\Lambda=\{\alpha_u|u^{\ell^2-\ell+1}=1\}$.

\begin{theorem}[\cite{segovia1}]
The group $\Lambda$ is a central subgroup of $\aut(\cX)$. The quotient curve $\cX/\Lambda$ is the Hermitian curve $\cH$ over $\fls$ with equation $X^{\ell+1}=Y^\ell+Y$.
The factor group $\aut(\cX)/\Lambda$ is isomorphic to $\PGU(3,\ell)$. 
\end{theorem}

If the non-degenerate Hermitian form in the three dimensional
vector space $V(3,\ell^2)$ over $\fls$ is given by
$Y^\ell T+YT^\ell-X^{\ell+1}$ then the unitary group $\U(3,\ell)$ is the
subgroup of $\GL(3,\ell^2)$ whose elements $U=(u_{ij})$ are
determined by the condition $U^{t}D\sigma(U)=D$ where $$D=\left(%
\begin{array}{ccc}
  -1 & 0 & 0 \\
  0 & 0 & 1 \\
  0 & 1 & 0 \\
\end{array}%
\right)
$$
and $\sigma(U)=(u_{ij}^\ell)$. $\U(3,\ell)$ has order
$(\ell+1)(\ell^3+1)\ell^2(\ell^2-1)$.
A diagonal matrix $[u,u,u]$ is in $\U(3,\ell)$ if and only if
$u^{\ell+1}=1$, and such matrices form a cyclic subgroup $C$ of
$\U(3,\ell)$.

The (normal) subgroup $\SU(3,\ell)$ is the subgroup of $\U(3,\ell)$ of
index $\ell+1$ consisting of all matrices with determinant $1$.
A set of generators of $\SU(3,\ell)$ are given by the following
matrices:

For $b,c\in \fls$ such that $c^\ell+c-b^{\ell+1}=0$, and for $a\in
\fls,\, a\neq 0$,
    $$ Q_{(b,c)}=
\left( \begin{array}{ccccc} 1 & 0 & b
\\ b^\ell & 1 &   c \\
0 & 0 & 1
\end{array} \right),\,R_a=
\left( \begin{array}{ccccc} a^{-n} & 0 & 0 \\ 0 & a^{n-1} & 0 \\
0 & 0 & a
\end{array} \right),\,
\, S=\left(\begin{array}{ccccc}  0 & 0 & 1 \\ 0 & -1& 0\\
1& 0 & 0
\end{array} \right)\, .
   $$
$\SU(3,\ell)\cap C$
is either trivial or is a subgroup of order $3$, according
as $\gcd(3,\ell+1)=1$ or $3$. 
The center
$Z(\U(3,\ell))$ coincides with $C$, and
$Z(\SU(3,\ell))=\SU(3,\ell)\cap C$. In this context,
$\PGU(3,\ell)=U(3,\ell)/C$.
A
treatise on unitary groups can be found in \cite[Section
10]{taylor}.

{}From each $U\in \U(3,\ell)$ a $(4\times 4)$-matrix $\widetilde{U}$
arises by adding $0,0,1,0$ as a third row and as a third column.
Obviously, these matrices $\widetilde{U}$ with $U\in \SU(3,\ell)$
form a subgroup $\Gamma$ of $\GL(4,\ell^2)$ isomorphic to $\SU(3,\ell)$.
Since the identity matrix is the only scalar matrix in $\Gamma$,
we can regard $\Gamma$ as a projective group in $\PGL(4,\fqq)$.

It is shown in \cite{segovia1} that the group $\Gamma$ preserves $\cX$,
$\Lambda$ centralizes $\Gamma$, and $\Gamma\cap \Lambda$ is
trivial when $\gcd(3,\ell+1)=1$ while it has order $3$ when
$\gcd(3,\ell+1)=3$.
Let $\Lambda_3$ be the unique subgroup of $\Lambda$ of order $\frac{\ell^2-\ell+1}{3}$.
Then by \cite[Lemma 8]{segovia1} $\aut(\cX)$ has a subgroup $\Xi$ with
\begin{equation}
\label{subgroupM} \Xi=
\begin{cases} \Gamma \times \Lambda,\qquad {\mbox{when}}\quad \gcd(3,\ell+1)=1;\\
\Gamma \times \Lambda_3,\qquad {\mbox{when}}\quad
\gcd(3,\ell+1)=3.\\
\end{cases}
\end{equation}
When $\gcd(3,\ell+1)=1$, $\aut(\cX)=\Xi$ holds (see \cite[Thm. 6 (i)]{segovia1}), whereas
for $\gcd(3,\ell+1)=3$, $\cX$ has further $\fqq$-automorphisms. 
Let $\rho$ be a primitive $(\ell^3+1)$-st root of unity in
$\fqq$, and let $\widetilde{E}$ be the diagonal matrix
$[\rho^{-1},\rho^{\ell^2-\ell},1,\rho^{-1}]$. Then $\widetilde{E}$
preserves $\cX$, normalizes $\Gamma$ and commutes with $\Lambda$. Moreover, $\widetilde{E}\notin \Xi$ but 
$\widetilde{E}^3\in \Xi$.
By \cite[Thm. 6 (ii)]{segovia1}, if $\gcd(3,n+1)=3$ then 
 $[\aut(\cX):\Xi]=3$ and $\Xi$ is a
normal subgroup of $\aut(\cX)$. Moreover, $\Gamma$ is a normal subgroup of $\aut(\cX)$.


Both $\Gamma$ and
$\Lambda$ preserve the set of points lying in the plane of
equation $Z=0$.
\begin{theorem}[\cite{segovia1}]
\label{orbits} The set of $\fqq$-rational points of $\cX$ splits
into two orbits under the action of $\aut(\cX)$, one is non-tame,
 has size $\ell^3+1$, and consists of the $\fls$-rational  points on $\cX$; the other is tame of size
$\ell^3(\ell^3+1)(\ell^2-1).$ Furthermore, $\aut(\cX)$ acts on the non-tame
orbit as $\PGU(3,\ell)$ in its doubly transitive permutation
representation.
\end{theorem}
Henceforth, the orbit of size $\ell^3+1$ will be denoted as ${\mathcal O}_1$, whereas the orbit of size $\ell^3(\ell^3+1)(\ell^2-1)$ by ${\mathcal O}_2$. Moreover, the natural projection 
from $\aut(\cX)$ to $PGU(3,\ell)$ will be denoted by $\pi$. Let $\phi$ be the rational map $\phi:\cX\to \cH$ defined by $\phi(1:x:y:z)=(1:x:y)$.

For a subgroup $L$ of $\aut(\cX)$, let ${\bar L}$ be the subgroup $\pi(L)$ of $PGU(3,\ell)$. 

Throughout the paper we will refer to the following maximal subgroups, defined up to conjugacy, of the group $PGU(3,\ell)$, viewed as the group of the projectivities of  ${\mathbb P}^2({\mathbb K})$ preserving the Hermitian curve $\cH$.
\begin{itemize}
\item[(A)]  The stabilizer of an $\fls$-rational point, of size $\ell^3(\ell^2-1)$.

\item[(B)] The normalizer of a Singer group, of size $3(\ell^2-\ell+1)$. Here a Singer group of $PGU(3,\ell)$ is a cyclic group of size $\ell^2-\ell+1$ stabilizing a point in $\cH(\fqq)\setminus \cH(\fls)$.

\item[(C)] The self-conjugate triangle stabilizer, of  size $6(\ell+1)^2$.

\item[(D)] The non-tangent line stabilizer, of size $\ell(\ell+1)(\ell^2-1)$.

\end{itemize}

\section{Preliminary results}\label{section3}

Let $L$ be a subgroup of $\aut(\cX)$. Let $\cX/L$ be the quotient curve of $\cX$ with respect to $L$, and let $g_L$ be its genus. If $L$ is tame, that is $p$ does not divide the order of $L$, then the Hurwitz genus formula for Galois extensions gives
\begin{equation}\label{HurTame}
(\ell^3+1)(\ell^2-2)=|L|(2g_L-2)+e_L
\end{equation}
with
\begin{equation}\label{HurTame2}
e_L=\sum_{P\in \cX}(|L_P|-1),
\end{equation}
where $L_P$ is the stabilizer of $P$ in $L$. The aim of this section is to provide a formula which relates $e_L$ to the action of ${\bar L}$ on the Hermitian curve $\cH$, see Proposition \ref{dielleprop} below.
Let $L_\Lambda=L\cap \Lambda$. The factor group $L/L_\Lambda$ is isomorphic to ${\bar L}$, and the action of $L/L_\Lambda$ on the orbits of $\cX$ under $\Lambda$ is isomorphic to that of ${\bar L}$ on $\cH$.

As to the relation between $e_L$ and the analogous value $
\sum_{\bar P \in \cH}(|\bar L_{\bar P}|-1),
$ for $\bar L$, 
by standard arguments from permutation group theory it is not difficult to prove that
\begin{equation}\label{bru}
\sum_{P\in \cX}(|L_P|-1)\cdot m_P=|L_\Lambda|\left(\sum_{\bar P \in \cH}(|\bar L_{\bar P}|-1)\cdot |\phi^{-1}(\bar P)|\right)-\sum_{P\in \cX}(m_P-|L_\Lambda|),
\end{equation}
where $m_P$ denotes the size of the orbit of $P$ under the action of the subgroup of $L$ stabilizing the set $\phi^{-1}(\phi(P))$.
However, we will not use \eqref{bru}, as this would require involved computations on $m(P)$. As it has emerged from the literature, a more adequate approach is based on the equality
\begin{equation}\label{dielle}
e_L=\sum_{h\in L,h\neq id}N_{h},\qquad \text{where } N_h=|\{P\in \cX\mid h(P)=P\}|
\end{equation}
(cf. \cite[Eq. 4.7]{garcia-stichtenoth-xing2001}).

The ramification points of the morphism $\phi:\cX\to \cH$ are exactly the points in ${\mathcal O}_1$. At these points $\phi$ is fully ramified. The set ${\bar {\mathcal O}_1}$ of the images of the points in ${\mathcal O}_1$ by $\phi$ in $\cH$ is precisely the set $\cH(\fls)$ of the $\fls$-rational points of $\cH$, whereas the image ${{\bar \cO}}_2$  of  ${\mathcal O}_2$ by $\phi$
coincides with  $ \cH(\fqs)\setminus \cH(\fls)$. Any point of $\cH$ fixed by some non-trivial element in $PGU(3,\ell)$ lies in ${{\bar \cO}}_1\cup {{\bar \cO}}_2$, see e.g. \cite[Prop. 2.2]{garcia-stichtenoth-xing2001}.

In order to compute $e_L$ as in \eqref{dielle}, it is convenient to write
\begin{equation}\label{separ}
N_{h}=N_{h}^{(1)}+N_{h}^{(2)}
\end{equation}
 with 
$$
 N_{h}^{(1)}=|\{P\in {\mathcal O}_1\mid h(P)=P\}|,\quad N_{h}^{(2)}=|\{P\in {\mathcal O}_2\mid h(P)=P\}|.
$$
\begin{proposition}\label{dielleprop} Let $L$ be a subgroup of $\aut(\cX)$, and let $L_\Lambda=L\cap \Lambda$. Let $\bar {\mathcal O}_1=\cH(\fls)$ and 
$\bar {\mathcal O}_2=\cH(\fqs)\setminus \cH(\fls)$.
Then
$$
e_L=(|L_\Lambda|-1)(\ell^3+1)+|L_\Lambda|n_1+|L_\Lambda|n_2
$$
where 
\begin{itemize}
\item $n_1$ counts the non-trivial relations $\bar h(\bar P)=\bar P$ with $\bar h \in \bar L$ when $\bar P$ varies in $\bar O_1$, namely $$n_1=\sum_{{\bar h}\in {\bar L}, {\bar h}\neq id}|\{{\bar P}\in \bar {\mathcal O}_1\mid {\bar h}({\bar P})={\bar P} \}|;$$
\item $n_2$ counts the non-trivial relations $\bar h(\bar P)=\bar P$ with $\bar h \in \bar L$ when $\bar P$ varies in $\bar O_2$, each counted with a multiplicity
$l_{{\bar h},\bar P}$ defined as the number of orbits of $\phi^{-1}(\bar P)$ under the action of $L_\Lambda$ that are fixed by an element $h\in \pi^{-1}(\bar h) $. That is, 
$$n_2=|L_\Lambda|\sum_{{\bar h}\in {\bar L}, {\bar h}\neq id}\,\,\,\sum_{{\bar P}\in \bar {\mathcal O}_2, {\bar h}({\bar P})={\bar P}}\,\,\,l_{{\bar h},{\bar P}}.$$
\end{itemize}
\end{proposition}
\begin{proof}
Note that
\begin{equation}\label{mah}
\sum_{h\in L,h\neq id}N_{h}=\sum_{k\in L_\Lambda, k\neq id}N_k+\sum_{\bar h \in \bar L, \bar h \neq id}\,\,\,\sum_{k\in L_\Lambda}N_{hk},
\end{equation}
where $h\in L$ is an element such that $\pi(h)=\bar h$.
For each non-trivial element $k\in L_\Lambda$,  $N_k=|{\mathcal O}_1|=(\ell^3+1)$ holds. Therefore 
\begin{equation}\label{boh}
\sum_{k\in L_\Lambda, k\neq id}N_k=(|L_\Lambda|-1)(\ell^3+1).
\end{equation}
As to the second term in the right hand side of \eqref{mah}, 
write 
$
N_{hk}=N_{hk}^{(1)}+N_{hk}^{(2)},
$
with $N_{hk}^{(i)}$ as in \eqref{separ}.
As $k(P)=P$ for each $P\in {\mathcal O}_1$, we have
\begin{equation}\label{prim}
\sum_{\bar h \in \bar L, \bar h \neq id}\,\,\,\sum_{k\in L_\Lambda} N_{hk}^{(1)}= |L_\Lambda|\sum_{\bar h \in \bar L, \bar h \neq id}|\{\bar P \in \bar {\mathcal O}_1 \mid \bar h(\bar P)=\bar P\}|.
\end{equation}
It remains to compute $\sum_{\bar h \in \bar L, \bar h \neq id}\sum_{k\in L_\Lambda} N_{hk}^{(2)}$. Since $\phi(k(P))=\phi(P)$ for each $P\in \cX$, condition $(hk)(P)=P$ yields that $\bar h(\phi(P))=\phi(P)$. Therefore, 
$$
\sum_{\bar h \in \bar L, \bar h \neq id}\,\,\,\sum_{k\in L_\Lambda} N_{hk}^{(2)}=\sum_{\bar h \in \bar L, \bar h \neq id}\,\,\,\sum_{\bar P \in \bar {\mathcal O}_2,\bar h(\bar P)=\bar P}\,\,\,\sum_{k\in L_\Lambda}m_{k,\bar h,\bar P}
$$
where $m_{k,\bar h,\bar P}=|\{P \in \cX, \pi(P)=\bar P, (hk)(P)=P\}|$. By the orbit-stabilizer theorem
$
\sum_{k\in L_\Lambda}m_{k,\bar h,\bar P}=|L_\Lambda|l_{\bar h,\bar P},
$
whence
$$\sum_{\bar h \in \bar L, \bar h \neq id}\,\,\,\sum_{k\in L_\Lambda} N_{hk}^{(2)}=|L_\Lambda|\sum_{{\bar h}\in {\bar L}, {\bar h}\neq id}\,\,\,\sum_{{\bar P}\in \bar {\mathcal O}_2, {\bar h}({\bar P})={\bar P}}l_{{\bar h},{\bar P}}.$$
Taking into account \eqref{dielle},\eqref{mah},\eqref{boh},\eqref{prim}, this finishes the proof.
\end{proof}
The following corollary to Proposition \ref{dielleprop} will be useful in the sequel.
\begin{proposition}\label{dielleprop2}
Let $L$ be a tame subgroup of $\aut(\cX)$. Assume that no non-trivial element in $\bar L$ fixes a point in $\cH\setminus \bar {\mathcal O}_1$. Then
$$
g_L=g_{\bar L}+\frac{(\ell^3+1)(\ell^2-|L_\Lambda|-1)-|L_\Lambda|(\ell^2-\ell-2)}{2| L|},
$$
where $g_{\bar L}$ is the genus of the quotient curve $\cH/{\bar L}$.
\end{proposition}
\begin{proof} By \eqref{HurTame} and Proposition \ref{dielleprop},
$$
(\ell^3+1)(\ell^2-2)=|L|(2g_L-2)+(|L_\Lambda|-1)(\ell^3+1)+|L_\Lambda|\sum_{{\bar h}\in {\bar L}, {\bar h}\neq id}|\{{\bar P}\in \bar {\mathcal O}_1\mid {\bar h}({\bar P})={\bar P} \}|.
$$
On the other hand, by the Hurwitz genus formula applied to the covering $\cH\to \cH/\bar L$
$$
\sum_{{\bar h}\in {\bar L}, {\bar h}\neq id}|\{{\bar P}\in \bar {\mathcal O}_1\mid {\bar h}({\bar P})={\bar P} \}|=(\ell^2-\ell-2)-|\bar L|(2g_{\bar L}-2), 
$$
whence
$$
(\ell^3+1)(\ell^2-2)=|L|(2g_L-2)+(|L_\Lambda|-1)(\ell^3+1)+|L_\Lambda|(\ell^2-\ell-2)-|L|(2g_{\bar L}-2).
$$
Then the claim follows by straightforward computation.
\end{proof}

When $\Lambda \subseteq L$, $l_{\bar h, \bar P}=1$ for every $\bar h\in \bar L $, and for every $\bar P \in \bar {\mathcal O}_2$ with $\bar h(\bar P)=\bar P$. Therefore Proposition \ref{dielleprop} reads as follows.
\begin{corollary}\label{diellecor}
Let $L$ be a subgroup of $\aut(\cX)$ containing $\Lambda$. 
Then
$$
e_L=(\ell^2-\ell)(\ell^3+1)+(\ell^2-\ell+1)\sum_{{\bar h}\in {\bar L}, {\bar h}\neq id}|\{{\bar P}\in \bar \cH(\fqs)\mid {\bar h}({\bar P})={\bar P} \}|.
$$
\end{corollary}
We end this section with a result showing that if $L$ is tame and $\Lambda\subset L$, then the genus of $g_L$ is actually  the genus of a quotient curve of the Hermitian curve $\cH$.
\begin{corollary}\label{diellecor2}
Let $L$ be a tame subgroup of $\aut(\cX)$ containing $\Lambda$. 
Then $g_L$ coincides with the genus of the quotient curve $\cH/\bar L$.
\end{corollary}
\begin{proof}
Let $g_\cH=\ell(\ell-1)/2$ be the genus of $\cH$, and let $g_{\bar L}$ be the genus of the quotient curve $\cH/{\bar L}$. Then by straightforward computation
\begin{equation}\label{cavendish1}
(\ell^3+1)(\ell^2-2)=(\ell^2-\ell+1)(2g_\cH-2)+(\ell^2-\ell)(\ell^3+1).
\end{equation}
Since ${\bar L}$ is tame, 
\begin{equation}\label{cavendish2}
2g_\cH-2= |\bar L|(2g_{\bar L}-2)+\sum_{{\bar h}\in {\bar L}, {\bar h}\neq id}|\{{\bar P}\in \bar \cH(\fqs)\mid {\bar h}({\bar P})={\bar P} \}|.
\end{equation}
Note that $|L|=(\ell^2-\ell+1)|\bar L|$. Taking into account \eqref{HurTame} and Corollary \ref{diellecor}, $g_L=g_{\bar L}$ follows by straightforward computation.
\end{proof}

%

\section{Curves $\cX/L$ with ${\bar L}$ subgroup of a group of type (A)}

In this section  subgroups $L$ of $\aut(\cX)$ stabilizing a point $P\in {\mathcal O}_1$ are investigated. Up to conjugacy, we may assume that $L$ is contained in the stabilizer $\aut(\cX)_{P_\infty}$ of  $P_\infty$ in $\aut(\cX)$. By the orbit-stabilizer theorem the size of $\aut(\cX)_{P_\infty}$ is $\ell^3(\ell^2-1)(\ell^2-\ell+1)$. Since $\aut(\cX)_{P_\infty}$ is non-tame,  in order to determine the genus of $\cX/L$  we will use Hilbert's ramification theory, see \cite[Ch. III.8]{stichtenothbook}.

Let $G$ be a subgroup of $\aut(\cX)$ and let $P$ be a point of $\cX$.  For an integer $i\ge -1$ the $i$-th ramification group of $G$ at $P$ is
$$
G_i(P)=\{h \in G\mid v_{P_\infty}(h^\star(t)-t)\ge i+1\},
$$
where $h^\star\in \aut(\K(\cX))$ is the pullback of $h$, $v_{P}$ is the discrete valuation of $\K(\cX)$ associated to $P$, and $t$ is any $P$-prime element. The group $G_0(P)$  coincides with the  stabilizer $\aut(\cX)_{P}$, whereas $G_1(P)$ is the only $p$-Sylow subgroup of $G_0(P)$, see e.g. 
\cite[Prop. III.8.5]{stichtenothbook}. Moreover, there exists a cyclic group $H$ in $G_0(P)$ such that $G_0(P)=G_1(P)\rtimes H$, see \cite[Thm. 11.49]{hirschfeld-korchmaros-torres2008}.
The Hurwitz genus formula together with the Hilbert different formula (see e.g. \cite[Thm. III.8.7]{stichtenothbook}) gives
\begin{equation}\label{huhi}
2g-2=|G|(2g_G-2)+\sum_{P\in \cX}d_P,\qquad \text{with }d_P=\sum_{i=0}^{\infty}(|G_i(P)|-1),
\end{equation}
where $g_G$ denotes the genus of the quotient curve $\cX/G$. 

Assume that $G=G_0(P_\infty)$, that is every element in $G$ fixes $P_\infty$. Since $\cX$ is a maximal curve,  no $p$-element in $G$ can fix a point $P$ different from $P_\infty$, see \cite[Thm. 9.76]{hirschfeld-korchmaros-torres2008} and \cite[Thm. 11.133]{hirschfeld-korchmaros-torres2008}. Therefore for any $P\neq P_\infty$ the integer $d_P$ in \eqref{huhi} coincides with $G_0(P)-1$. The following result then holds.

\begin{lemma}\label{stab11} For a subgroup $L$ of $\aut(\cX)_{P_\infty}$,  let $g_L$ be the genus of the quotient curve $\cX/L$. Then
$$(\ell^3+1)(\ell^2-2)=|L|(2g_L-2)+e_L+\sum_{i=1}^{\infty}(|L_i(P_\infty)|-1),$$
with $e_L$ as in \eqref{HurTame2}.
\end{lemma}

We now provide an explicit description of $\aut(\cX)_{P_\infty}$. For $a\in \fqs$ such that $a^{(\ell^2-\ell+1)(\ell^2-1)}=1$, for $b,c \in \fls$ such that $c^\ell+c=b^{\ell+1}$, let $\xi_{a,b,c}$ be the projectivity in $PGL(4,q^2)$ defined by the matrix
$$
\begin{pmatrix}
a^{\ell^2-\ell+1} & 0 & 0 & b \\
b^\ell a^{\ell^2-\ell+1} & a^{\ell^3+1} & 0 & c \\
0 & 0 & a^{\ell^2}  & 0 \\
0 & 0 & 0 & 1
\end{pmatrix}.
$$
It is easily seen that $\xi_{a,b,c}=\xi_{a,0,0}\xi_{1,b,c}$, with $\xi_{1,b,c}\in \Gamma\cap \aut(\cX)_{P_\infty}$. Also, by straightforward computation $\xi_{a,0,0}$ lies in $\aut(\cX)_{P_\infty}$. By a trivial counting argument, we then have 
$$
\aut(\cX)_{P_\infty}=\{\xi_{a,b,c}\mid a^{(\ell^2-\ell+1)(\ell^2-1)}=1, b,c\in \fls, b^{\ell+1}=c^\ell+c\}.
$$
The elements $\xi_{1,b,c}$ form a subgroup of $\aut(\cX)_{P_\infty}$ of size $\ell^3$; therefore  the first ramification group of $\aut(\cX)_{P_\infty}$ at $P_\infty$ is
$$
\{\xi_{1,b,c}\mid b,c\in \fls, b^{\ell+1}=c^\ell+c\}.
$$
In order to determine higher ramification groups at $P_\infty$ we need to compute the integer $v_{P_\infty}(h^\star(t)-t)$, with $t$ a $P_\infty$-prime element, for automorphisms $h=\xi_{1,b,c}$. By \cite[Sect. 4]{segovia1} 
$$
v_{P_\infty}(x)=-(\ell^3-\ell^2+\ell),\quad
v_{P_\infty}(y)=-(\ell^3+1),\quad
v_{P_\infty}(z)=-\ell^3.
$$
Therefore, $t$ can be assumed to be the rational function $z/y$. Since
$$
\xi_{1,b,c}^\star(z)=z,\qquad \xi_{1,b,c}^\star(y)=y+b^\ell x+c
$$
we have that
$$
\xi_{1,b,c}^\star(z/y)-(z/y)=\frac{c-zxb^\ell}{y(y+b^\ell x+c)},
$$
whence
$$
v_{P_\infty}(\xi_{1,b,c}^\star(t)-t)=\left\{
\begin{array}{ll}
\ell^2-\ell+2 & \text{if } b\neq 0\\
\ell^3+2 & \text{if } b=0
\end{array}
\right.
\,.$$
The following result is then obtained.
\begin{proposition}\label{stabi2} Let $L$ be a subgroup of $\aut(\cX)_{P_\infty}$. Then
$$
L_1(P_\infty)=L_2(P_\infty)=\ldots=L_{\ell^2-\ell+1}(P_\infty)=\{\xi_{1,b,c}\mid \xi_{1,b,c}\in L\},
$$
and
$$
L_{\ell^2-\ell+2}(P_\infty)=L_{\ell^2-\ell+3}(P_\infty)=\ldots= L_{\ell^3+1}(P_\infty)=\{\xi_{1,0,c}\mid \xi_{1,0,c}\in L\}.
$$
For $i>\ell^3+1$ the group $L_i(P_\infty)$ is trivial.
\end{proposition}
As to the computation of $e_L$ in Lemma \ref{stab11}, the following fact will be useful.
\begin{lemma}\label{obv} Let $L$ be a subgroup of $\aut(\cX)_{P_\infty}$. Then any point of $\cX$ which is fixed by a non-trivial element in $L$ belongs to ${\mathcal O}_1$.
\end{lemma}
\begin{proof}
Assume that $\alpha\in L$ fixes a point $P\in \cX\setminus {\mathcal O}_1$. Then $\pi(\alpha)$ is an element of $PGU(3,\ell)$ fixing both the infinite point $\bar P_\infty$ of $\cH$ and the point $\phi(P)$ in $\cH\setminus \cH(\fls)$. Then by \cite[Sect. 4]{garcia-stichtenoth-xing2001} $\pi(\alpha)$ is trivial, that is $\alpha\in \Lambda$. Since any non-trivial element in $\Gamma$ only fixes points in ${\mathcal O}_1$, we obtain $\alpha=id$. 
\end{proof}

In Section \ref{par(A)} we will deal with the case $L=\Sigma_1\times \Sigma_2$, where $\Sigma_1$ is contained in $\Gamma$ and $\Sigma_2$ is a subgroup of $\Lambda$, see Section \ref{par(A)}. To this end, we determine a subgroup $\Omega$ of $\Gamma\cap \aut(\cX)_{P_\infty}$ such that  $\Omega\cap \Lambda=\{id\}$. In Section \ref{gen(A)} the case $L=\pi^{-1}(\bar G)$ with $\bar G$ a group of type (A) will be dealt with.

Let ${\mu_1}$ be the highest power of $3$ dividing $\ell+1$. Let
$$
\Omega=\{\xi_{a,b,c}\mid a^\frac{\ell^2-1}{{\mu_1}}=1\}.
$$
Assume that $\alpha=\xi_{a,b,c}\in \Omega \cap \Lambda$. Then clearly $b=c=0$ holds, whence $\alpha=\xi_{a,0,0}$ for some $a$ such that $a^{(\ell^2-1)/{\mu_1}}=1$. Note that $\xi_{a,0,0}\in \Lambda$ implies $a^{\ell^2-\ell+1}=1$. But then $a=1$ since $\gcd(\frac{\ell^2-1}{{\mu_1}},\ell^2-\ell+1)=1$. Then the following holds.
\begin{lemma}
The subgroup generated by $\Omega$ and $\Lambda$ is the direct product $\Omega\times \Lambda$. The projection $\pi(\Omega\times \Lambda)$ is a subgroup of the stabilizer of the infinite point $\bar P_\infty$ of $\cH$ in $PGU(3,\ell)$ isomorphic to $\Omega$.
\end{lemma}

\subsection{$L=\Sigma_1\times \Sigma_2$, $\Sigma_1< \Omega$, $\Sigma_2<\Lambda$}\label{par(A)}
Let $|\Sigma_1|=mp^{v+w}$, where
$$
p^{v+w}=|\Sigma_1\cap \{\xi_{1,b,c}\}|,\qquad p^v=|\Sigma_1\cap \{\xi_{1,0,c}\}|,
$$
and $m$ is a divisor of $(\ell^2-1)/{\mu_1}$. Let $|\Sigma_2|=d_2$. Then by Lemma \ref{stab11}, together with Proposition \ref{dielleprop} and Lemma \ref{obv}, 
it follows that
$$
\begin{array}{rl}
\ell^5-2\ell^3+\ell^2-2=  &md_2p^{v+w}(2g_L-2)+(\ell^2-\ell+1)(p^{v+w}-1)+(\ell^3-\ell^2+\ell)(p^w-1)\\{} & +(d_2-1)(\ell^3+1)+d_2\sum_{\bar h \in \pi(\Sigma_1),\bar h \neq id}|\{\bar P\in \bar {\mathcal O}_1\mid \bar h(\bar P)=\bar P\}|,
\end{array}
$$
that is
\begin{equation}\label{10set}
\begin{array}{rl}
(2g_L-1)md_2p^{v+w}= & (\l^3+1)(\l^2-d_2)-(\ell^2-\ell+1)p^w(p^v+\ell)+d_2\\ {} & -d_2\sum_{\bar h \in \pi(\Sigma_1),\bar h \neq id}|\{\bar P\in \bar {\mathcal O}_1, \bar P \neq \bar P_\infty \mid \bar h(\bar P)=\bar P\}|.
\end{array}
\end{equation}
In order to provide concrete values of genera $g_L$ we are going to consider subgroups of $PGU(3,\ell)_{\bar P_\infty}$ that are known in the literature, see \cite{garcia-stichtenoth-xing2001} and \cite{abdon-quoos2004}.
To this end it is useful to note that in both papers \cite{garcia-stichtenoth-xing2001} and \cite{abdon-quoos2004} the group $PGU(3,\ell)_{\bar P_\infty}$ is denoted as $\cA(P_\infty)$, and that the subgroup $\pi(\Omega)$ of $PGU(3,\ell)$ in the notation of both \cite{garcia-stichtenoth-xing2001} and \cite{abdon-quoos2004} is the subgroup of index ${\mu_1}$ in $\cA(P_\infty)$ consisting of elements $[a,b,c]$ with $a^{(\ell^2-1)/{\mu_1}}=1$. It should also be noted that for each $\Sigma_1<\Omega$ the integer 
$\sum_{\bar h \in \pi(\Sigma_1),\bar h \neq id}|\{\bar P\in \bar {\mathcal O}_1, \bar P \ne \bar P_\infty\mid \bar h(\bar P)=\bar P\}|$ is computed in \cite{garcia-stichtenoth-xing2001}. From \eqref{10set}, together with Theorem 4.4 in \cite{garcia-stichtenoth-xing2001},it follows that
$$
g_L=\frac{\l^5+\l^2-(\ell^2-\ell+1)p^w(p^v+\ell)-d_2(\l^3+(d-1)p^v\l-dp^{v+w})}{2md_2p^{v+w}},
$$
where $d=\gcd(m,\ell+1)$.

The following result then holds.
\begin{theorem}\label{thmA1} Let $d_2$ be any divisor of $\ell^2-\ell+1$.
\begin{enumerate}
\item[\rm{(i)}] Let $p\neq2$ and $m\mid \l ^{2}-1$ be such that $3$ does not divide $m$ and $m>1$. Let $d=\gcd(m, \l+1)$ and let $s:=min\left\{r\geq1 : p^{r}\equiv1 \textrm{ mod } \frac{m}{d}\right\}$. For each $0\leq w\leq h$, such that $s\mid w$, there exists a subgroup $L$ of $\Aut(\cX)$ with
$$g_{L}=\frac{(\l^{3}+1)(\l^{2}-p^{w})-(\l^{3}-\l)d_{2}-dd_{2}(\l-p^{w})}{2md_2p^{w}}$$
(cf. \cite[Prop. 4.6]{garcia-stichtenoth-xing2001}).

\item[\rm{(ii)}] Let $p\neq2$ and $m\mid \l-1$. Let $d=\gcd(m, \l+1)$. Let $s$ be the order of $p$ in $(\ZZ/m\ZZ)^{*}$, and let $$r=\displaystyle{\left\{\begin{array}[pos]{lll} \textrm{order of p in }  (\ZZ/\frac{m}{2}\ZZ)^{*}, && m \textrm{ even}\\
s, && m \textrm{ odd}
\end{array}\right.}.$$ For  each $0\le v\le h$ such that $s\mid v$, and for each $0\leq w\leq h-1$ such that $r\mid w$, there exists a subgroup $L$ of $\Aut(\cX)$ with
$$g_{L}=\frac{\l^{5}+\l^{2}-\l^{3}d_{2}-(\l^{2}-\l+1-dd_{2})p^{v+w}-(\l^{3}-\l^{2}+\l)p^{w}-d_{2}\l p^{v}(d-1)}{2md_2p^{v+w}}$$
(cf. \cite[Thm. 1]{abdon-quoos2004}).

\item[\rm{(iii)}] Let $p\neq2$ and $m\mid \l^{2}-1$ be such that $m$ does not divide $\l-1$, and $3$ does not divide $m$. Let $d=\gcd(m, \l+1)$. Let $s$ be the order of $p$ in $(\ZZ/m\ZZ)^*$, and $r$ be the order of $p$ in $(\ZZ/(m/d)\ZZ)^*$. For each $0\leq v\leq h$, such that $v\mid 2h$, $v$ does not divide $h$ and $s\mid v$, and for each $\frac{v}{2}\leq w\leq h$, such that $r\mid w$, there exists a subgroup $L$ of $\Aut(\cX)$ with
$$g_{L}=\frac{\l^{5}+\l^{2}-\l^{3}d_{2}-(\l^{2}-\l+1-dd_{2})p^{v+w}-(\l^{3}-\l^{2}+\l)p^{w}-d_{2}\l p^{v}(d-1)}{2md_2p^{v+w}}$$
(cf. \cite[Thm. 2]{abdon-quoos2004}).

\item[\rm{(iv)}] Let $p\neq 2$. For each $0\leq v\leq h$, and  for each $0\leq w\leq h-1$, there exists a subgroup $L$ of $\Aut(\cX)$ with
$$g_{L}=\frac{\l^{5}+\l^{2}-(\l^{2}-\l+1)p^{w}(p^{v}+\l)+\l d_{2}(p^{v}-\l^{2})-d_{2}p^{v}(\l-p^{w})}{2d_2p^{v+w}}$$
(cf. \cite[Thm. 3.2]{garcia-stichtenoth-xing2001}).


\item[\rm{(v)}] Let $p=2$. For all integers $v,w$ with $0\le v\le w<h$ there exists a subgroup $L$ of $\aut(\cX)$ with
$$
g_L=\frac{\l^{5}+\l^{2}-(\l^{2}-\l+1)2^{w}(2^{v}+\l)+\l d_{2}(2^{v}-\l^{2})-d_{2}2^{v}(\l-2^{w})}{d_22^{v+w+1}}
$$
(cf. \cite[Cor. 3.4(ii)]{garcia-stichtenoth-xing2001}).

\item[\rm{(vi)}] Let $p=2$. 
 For all integers $v,w$ with 
 $w\mid h$, $w\mid v$, $v\mid 2h$, $1\le v<n$, and $(2^v-1)/(2^w-1)\mid (2^h+1)$,
  there exist subgroups $L$ of $\aut(\cX)$ with
$$
g_L=\frac{\l^{5}+\l^{2}-(\l^{2}-\l+1)2^{w}(2^{v'}+\l)+\l d_{2}(2^{v'}-\l^{2})-d_{2}2^{v'}(\l-2^{w})}{d_22^{v'+w+1}}
$$
for each $v'$ with $0\le v'\le v$.
(cf. \cite[Cor. 3.4(i), Cor. 3.4(iii)]{garcia-stichtenoth-xing2001}).

\item[\rm{(vii)}] Let $p=2$ and $h$ be odd. Let $s\mid h$ and $0\leq k\leq s$. For each $1\leq v\leq h-1$, such that $v=s+k$, and for each $s\leq w\leq h-1$, there exists a subgroup $L$ of $\Aut(\cX)$ with
$$g_{L}=\frac{\l^{5}+\l^{2}-(\l^{2}-\l+1)2^{w}(2^{v}+\l)+\l d_{2}(2^{v}-\l^{2})-d_{2}2^{v}(\l-2^{w})}{d_22^{v+w+1}}$$
(cf. \cite[Thm. 4]{abdon-quoos2004}).

\item[\rm{(viii)}] Let $p=2$ and $h$ be even and such that $4$ does not divide $h$. Let $s\mid h$ be odd and $0\leq k\leq s$. For each $1\leq v\leq h-1$, such that $v=2s+k$, and for each $2s\leq w\leq h-1$, there exists a subgroup $L$ of $\Aut(\cX)$ with
$$g_{L}=\frac{\l^{5}+\l^{2}-(\l^{2}-\l+1)2^{w}(2^{v}+\l)+\l d_{2}(2^{v}-\l^{2})-d_{2}2^{v}(\l-2^{w})}{d_22^{v+w+1}}$$
(cf. \cite[Thm. 5]{abdon-quoos2004}).

\item[\rm{(ix)}] Let $p=2$ and write $h=2^{e}f$, with $e,f\in\NN$ and $f\geq 3$ odd. For each divisor $j$ of $f$, let $k_{j}$ be the order of $2$ in $(\ZZ/j\ZZ)^{*}$ and $r_{j}=\frac{\Phi(j)}{k_{j}}$, where $\Phi$ is the Euler function. For each $1\leq w\leq h-2$, such that $w=2^{e}\left[1+\sum_{\frac{j}{f}\neq1}l_{j}k_{j}\right]$, with $0\leq l_{j}\leq r_{j}$, there exists a subgroup $L$ of $\Aut(\cX)$ with
$$g_{L}=\frac{\l^{5}+\l^{2}-(\l^{2}-\l+1)2^{w}(2^{w+1}+\l)+\l d_{2}(2^{w+1}-\l^{2})-d_{2}2^{w+1}(\l-2^{w})}{d_22^{2w+2}}$$
(cf. \cite[Thm. 6]{abdon-quoos2004}).

\end{enumerate}

\end{theorem}

\subsection{$L=\pi^{-1}(\bar G)$, $\bar G<PGU(3,\ell)_{\bar P_\infty}$}\label{gen(A)} Groups $L=\pi^{-1}(\bar G)$ that have not already been considered in Section \ref{par(A)}  are groups $\pi^{-1}(\bar G)$ with $\bar G$ containing elements $[a,b,c]$ with $a^{(\ell^2-1)/\mu_1}\neq 1$ (again the notation of both \cite{garcia-stichtenoth-xing2001} and \cite{abdon-quoos2004} is used to describe elements in $PGU(3,\ell)_{\bar P_\infty}$). Let $|\bar G|={\bar m}p^{v+w}$, with $\gcd(\bar m, p)=1$, and $p^v=|\{[a,b,c]\in G\mid b=0\}|$.
Then it is easily seen that $|L|=mp^{v+w}$ with $m=(\ell^2-\ell+1)\bar m$, and
$$
p^{v+w}=|L\cap \{\xi_{1,b,c}\}|,\qquad p^v=|L\cap \{\xi_{1,0,c}\}|.
$$
Clearly, $L_\Lambda=\Lambda$ holds.

By Lemma \ref{stab11}, together with Corollary \ref{diellecor} and Lemma \ref{obv}, 
it follows that
$$
\begin{array}{rl}
\ell^5-2\ell^3+\ell^2-2=  &mp^{v+w}(2g_L-2)+(\ell^2-\ell+1)(p^{v+w}-1)+(\ell^3-\ell^2+\ell)(p^w-1)\\{} & +(\ell^2-\ell)(\ell^3+1)+(\ell^2-\ell+1)\sum_{\bar h \in G,\bar h \neq id}|\{\bar P\in \bar {\mathcal O}_1\mid \bar h(\bar P)=\bar P\}|.
\end{array}
$$
In order to provide concrete values of genera $g_L$ we are going to consider subgroups $\bar G$
containing elements $[a,b,c]$ with $a^{(\ell^2-1)/\mu_1}\neq 1$ that have been described in either 
\cite{garcia-stichtenoth-xing2001} or \cite{abdon-quoos2004}.
\begin{theorem}\label{thmA2}
\begin{enumerate}
\item[\rm{(i)}] Let $p\neq2$, $\bar{m}$ be an integer such that $\bar{m}\mid \l ^{2}-1$ and $3\mid\bar{m}$. Let $d=\gcd(\bar{m}, \l+1)$ and let $s:=min\left\{r\geq1 : p^{r}\equiv1 \textrm{ mod } \frac{\bar{m}}{d}\right\}$. For each $0\leq w\leq h$, such that $s\mid w$, there exists a subgroup $L$ of $\Aut(\cX)$ with
$$g_{L}=\frac{(\l-p^{w})(\l+1-d)}{2\bar{m}p^{w}}$$
(cf. \cite[Prop. 4.6]{garcia-stichtenoth-xing2001}).

\item[\rm{(ii)}] Let $p\neq2$ and $\bar{m}\mid \l^{2}-1$ be such that $\bar{m}$ does not divide $\l-1$, and $3\mid\bar{m}$. Let $d=\gcd(\bar{m}, \l+1)$. Let $s$ be the order of $p$ in $(\ZZ/\bar{m}\ZZ)^{*}$,  and let $$r=\displaystyle{\left\{\begin{array}[pos]{lll} \textrm{order of p in }  (\ZZ/\frac{\bar{m}}{2}\ZZ)^{*}, && \bar{m} \textrm{ even}\\
s, && \bar{m} \textrm{ odd}
\end{array}\right.}.$$
For each $0\leq v\leq h$, such that $v\mid 2h$, $v$ does not divide $h$ and $s\mid v$, and for each $\frac{v}{2}\leq w\leq h$ such that $r\mid w$, there exists a subgroup $L$ of $\Aut(\cX)$ with
$$g_{L}=\frac{(\l^{2}-p^{v+w}-\l p^{w}-d\l p^{v}+d p^{v+w}+\l p^{v})}{2\bar{m}p^{v+w}}$$
(cf. \cite[Thm. 2]{abdon-quoos2004}).

\end{enumerate}

\end{theorem}

\section{Curves $\cX/L$ with ${\bar L}$ subgroup of a group of type (B)}

In this  section we investigate subgroups $L$ of $\aut(\cX)$ with $L=\Sigma_1\times \Sigma_2$, where $\Sigma_1$ is contained  $\Gamma$, $\Sigma_2$ is a subgroup of $\Lambda$, and ${\bar L}$ is a subgroup of a group of type (B). 
To this end, we determine a subgroup $\Omega$ of $\Gamma$ such that $\pi(\Omega)$ is contained in a group of type (B) and $\Omega\cap \Lambda=\{id\}$. 

The construction of $\Omega$ requires some technical preliminaries, especially on the Singer groups of $PGU(3,\ell)$. As in \cite{cossidente-korchmaros-torres1999}, we use a representation of a Singer group up to conjugation in $GL(3,q^2)$. This will allow us to deal with diagonal matrices, thus avoiding more involved matrix computation.

By \cite[Prop. 4.6]{cossidente-korchmaros-torres1999} there exists a matrix $A_1$ in $GL(3,q^2)\setminus GL(3,\ell^2)$ such that $A_1^tD\sigma(A_1)=D_1$, with
$$
D_1=\begin{pmatrix}0 & 1 & 0 \\ 0 & 0 & 1 \\ 1 & 0 & 0\end{pmatrix}.
$$
Then  a matrix $M$ is in $SU(3,\ell)$ if and only if $M_1=A_1^{-1}MA_1$ is such that 
\begin{equation}\label{cond1}
M_1^tD_1\sigma(M_1)=D_1,\quad \det(M_1)=1.
\end{equation}
Also, it is straightforward to check that the points of ${\mathbb P}^2(\K)$ whose homogeneous coordinates are the columns of $A_1$ lie in $\cH(\fqq)\setminus \cH(\fls)$.

Now we construct a group $\Omega_1$ of $GL(3,q^2)$, contained in the conjugate subgroup of $SU(3,\ell)$ by $A_1$.
Let ${\mu_2}=\gcd(\ell^2-\ell+1,3)$. It is easily seen that either ${\mu_2}=3$ or ${\mu_2}=1$ according to whether $\ell \equiv 2\pmod 3$ or not. Let $\Pi_{\frac{\ell^2-\ell+1}{{\mu_2}}}$ be the group of $(\frac{\ell^2-\ell+1}{{\mu_2}})$-th roots of unity. As $\gcd(\frac{\ell^2-\ell+1}{{\mu_2}},\ell+1)=1$, for each $\lambda \in \Pi_{\frac{\ell^2-\ell+1}{{\mu_2}}}$ there exists a unique $\tilde{\lambda}\in \Pi_{\frac{\ell^2-\ell+1}{{\mu_2}}}$ with $\tilde{\lambda}^{\ell+1}=\lambda^{-\ell}$.
Let $\Omega_1$ be the  group generated by $D_1$ and 
$$
T=\left(
\begin{array}{ccc}
\tilde{w} w & 0 & 0 \\
0 & \tilde{w} w^\ell & 0 \\
0 & 0 & \tilde{w}
\end{array}
\right),
$$
where $w$ is a primitive $(\frac{\ell^2-\ell+1}{{\mu_2}})$-th root of unity.

Let $\Theta=<T>$ and  $\Upsilon_1=<D_1>$. 
It is straightforward to check that $\Omega_1=\Theta\rtimes \Upsilon_1$. Also, every matrix $M_1$ in $\Omega_1$ satisfies \eqref{cond1}, and thus for each $M_1\in \Omega_1$ the matrix $A_1M_1A_1^{-1}$ belongs to $SU(3,\ell)$. 

For a matrix $M_1\in \Omega_1 $ let $\epsilon(M_1)$ be the projectivity in $PGL(4,\ell^2)$ defined by the $4\times 4$ matrix obtained from $A_1M_1A_1^{-1}$ by adding $0,0,1,0$ as a third row and as a third column. Then $\epsilon: \Omega_1\to \Gamma$ is an injective group homomorphism. Let $\Omega=\epsilon(\Omega_1)$. 

Let ${\bar P}_i$ be the point of $\cH(\fqs)\setminus \cH(\fls)$ whose homogeneous coordinates are the elements in the $i$-th column of $A_1$. Then the subgroup $\pi(\epsilon(\Theta))$ is contained in the stabilizer of ${\bar P}_i$ in $PGU(3,\ell)$. Also, the group $\pi(\epsilon(\Upsilon_1))$  acts regularly on  $\{{\bar P}_1,{\bar P}_2, {\bar P}_3\}$.  

We now prove that $\Omega\cap \Lambda=\{id\}$. Let $\alpha\in \Omega \cap \Lambda$.  Since $\pi(\alpha)$ fixes every point in $\cH$, $\alpha \in \epsilon(\Theta)$. Taking into account that every non-trivial element in $\Gamma\cap \Lambda$ has order $3$, and that $3$ does not divide the order of $T$, we obtain that $\alpha=1$. 

The following result then holds.
\begin{lemma}
The subgroup generated by $\Omega$ and $\Lambda$ is the direct product $\Omega\times \Lambda$. The projection $\pi(\Omega\times \Lambda)$ is a subgroup of a group of type {\rm{(B)}} isomorphic to $\Omega_1$.
\end{lemma}

Let ${\bar P}_1=(x_i,y_i,1)$. As ${\bar P}_1\in \cH(\fqq)\setminus \cH(\fls)$, up to a rearrangement of the indexes we can assume that $x_2=x_1^{\ell^2}, y_2=y_1^{\ell^2}$, and 
$x_3=x_1^{\ell^4}, y_3=y_1^{\ell^4}$.

For any point $Q_i=(x_i,y_i,z_0,1)$  of $\cX$ such that $\phi(Q_i)=\bar P_i$, 
the image $\epsilon(T)(Q_i)$ of $Q_i$ by $\epsilon(T)$ is the point $(x_i,y_i,\frac{z_0}{a_{31}x_i+a_{32}y_i+a_{33}}, 1),$ where
$(a_{31},a_{32},a_{33})$ is the third row of the matrix $ A_1TA_1^{-1}$.
Let $s_i=a_{31}x_i+a_{32}y_i+a_{33}$. Since $T$ has order $(\ell^2-\ell+1)/\mu_2$, we have that $s_i^{(\ell^2-\ell+1)/{\mu_2}}=1$.
Note that $s_1^{\ell^2}=s_2$ and $s_1^{\ell^4}=s_3$ hold.
%
%

\begin{lemma} Any $s_i$ is a primitive $\frac{\ell^2-\ell+1}{{\mu_2}}$-th root of unity.
\end{lemma}
\begin{proof} It is enough to prove the assertion for $i=1$. Assume that $s_1^m=1$. Then also $s_2^m=1$ and $s_3^m=1$ hold. Therefore, $\epsilon(T)^m$ fixes every point $Q$ such that $\phi(Q)={\bar P}_i$ for some $i=1,2,3$. Since $\ell^2-\ell+1>3$,   the three lines joining $P_\infty$ to the points $(x_i,y_i,0,1)$,  $i=1,2,3$, are fixed by $\epsilon(T)^m$ pointwise. As these three lines are not coplanar,  $\epsilon(T)^m$ is  the identical projectivity of $PG(3,{\mathbb K})$. This shows that $\ell^2-\ell+1$ divides $m$, and the proof is complete.
\end{proof}

%
%

Subgroups of the normalizers of a Singer group of $PGU(3,n)$ have been classified up to conjugacy in \cite[Chapter 4]{SHORT}, see also \cite[Lemma 4.1]{cossidente-korchmaros-torres2000}. As a straightforward consequence, the following result holds.
\begin{lemma}\label{lisst} The following is a complete list of subgroups of $\Omega$, up to conjugacy.
\begin{itemize}
\item[(a)] For every divisor $d$ of $(\ell^2-\ell+1)/{\mu_2} $, the cyclic subgroup of $\epsilon(\Theta)$ of order $d$, i.e. the subgroup generated by $\epsilon(T)^{(\ell^2-\ell+1)/d{\mu_2}}$.
\item[(b)] For every divisor $d$ of $(\ell^2-\ell+1)/{\mu_2} $, the subgroup of order $3d$ which is the semidirect product of 
the cyclic subgroup of $\epsilon(\Theta)$ of order $d$ with $\epsilon(\Upsilon_1)$.
\end{itemize}
\end{lemma}
We deal separately with cases (a) and (b) of Lemma \ref{lisst}.

\subsection{$L=\Sigma_1\times \Sigma_2$, $\Sigma_1< \epsilon(\Theta)$, $\Sigma_2<\Lambda$}\label{jji}
Let $\Sigma_1=<\epsilon(T)^{i_1}>$, with $i_1= (\ell^2-\ell+1)/d{\mu_2}$, and let $\Sigma_2$ be the group generated by the projectivity defined by the diagonal matrix $[1,1,\beta^{i_2},1]$, with $\beta$ a primitive $(\ell^2-\ell+1)$-th root of unity and $i_2$ a divisor of $\ell^2-\ell+1$. 
Let $d_2=(\ell^2-\ell+1)/i_2$.
Without loss of generality assume that $s_1=\beta^{3^k}$, with $k=0$ when ${\mu_2}=1$ and $k>0$ for ${\mu_2}=3$. 
In order to compute integers $l_{\bar h,\bar P}$ as in Proposition \ref{dielleprop}, we need to investigate the action of $\Sigma_1$ on the orbits of $\phi^{-1}({\bar P}_i)$ under  $\Sigma_2$. Fix $Q_1=(x_1,y_1,z_0,1)\in \phi^{-1}({\bar P}_1)$.
The orbits of $\phi^{-1}({\bar P}_1)$ under the action of $\Sigma_2$ are 
$$
\Delta_j=\{(x_1,y_1,\beta^jz_0,1),(x_1,y_1,\beta^{j+i_2}z_0,1),\ldots, (x_1,y_1,\beta^{j+(d_2-1)i_2}z_0,1)\},$$
with  $j=0,\ldots,i_2-1$.
For $1\le t \le d-1$, the orbit $\Delta_j$ is fixed by $\epsilon(T)^{ti_1}$ if and only if 
$$
\beta^js_1^{-ti_1}=\beta^{j+ui_2},\quad \text{for some }u \in {\mathbb Z},
$$
that is,
%
$$
3^kti_1= v(\ell^2-\ell+1 )-ui_2, \quad \text{for some }u, v\in {\mathbb Z}.
$$
Equivalently,
\begin{equation}\label{divi}
i_2| 3^kti_1.
\end{equation}
Similarly it can be proved that $\epsilon(T)^{ti_1}$ fixes any orbit of $\phi^{-1}({\bar P}_2)$ (resp. $\phi^{-1}({\bar P}_3)$) under $\Sigma_2$
if and only if $i_2| 3^k\ell^2ti_1\quad  \text{(resp. }i_2| 3^k\ell^4ti_1 \text{)}.$
Since $\gcd(i_2,\ell)=1$,  either $\epsilon(T)^{ti_1}$ fixes all the orbits of $\phi^{-1}(\{{\bar P}_1 ,{\bar P}_2, {\bar P}_3\})$ under $\Sigma_2$ or none, according to whether \eqref{divi} holds or not. 

If $3$ does not divide $i_2$, then $\eqref{divi}$ is equivalent to $i_2| ti_1$. Therefore, the number of non-trivial  elements in $\Sigma_1$ fixing one (and hence every) orbit
of $\phi^{-1}(\{{\bar P}_1 ,{\bar P}_2, {\bar P}_3\})$ under $\Sigma_2$ is equal to the number of common multiples of $i_1$ and $i_2$ that are strictly less than $(\ell^2-\ell+1)/{\mu_2}$. If ${\rm{lcm}}(i_1,i_2)\le (\ell^2-\ell+1)/{\mu_2}$, then this number is $\frac{\ell^2-\ell+1}{{\mu_2}{\rm{lcm}}(i_1,i_2)}-1$; if ${\rm{lcm}}(i_1,i_2)= \ell^2-\ell+1$, then no orbit is fixed by a non-trivial element in $\Sigma_1$.

If $3$ divides $i_2$, then $\eqref{divi}$ is equivalent to $\frac{i_2}{3}| ti_1$. Arguing as in the previous case, it can be deduced that 
the number of non-trivial  elements in $\Sigma_1$ fixing one (and hence every) orbit
of $\phi^{-1}(\{{\bar P}_1 ,{\bar P}_2, {\bar P}_3\})$ under $\Sigma_2$ is  $\frac{\ell^2-\ell+1}{{\mu_2}{\rm{lcm}}(i_1,i_2/3)}-1$.

The last term of $e_L$ as in Proposition \ref{dielleprop}, can be written as
$$
|L_\Lambda|n_2=|\Sigma_2|\sum_{{\bar h}\in {\bar \Sigma_1}, {\bar h}\neq id}\,\,\,\sum_{{\bar P}\in \{\bar P_1,\bar P_2,\bar P_3\},{\bar h}({\bar P})={\bar P}}\,\,\,l_{{\bar h},{\bar P}};
$$
then, it is equal to 
$$
\left\{
\begin{array}{ll}
3(\ell^2-\ell+1)(\frac{\ell^2-\ell+1}{{\mu_2}{\rm{lcm}}(i_1,i_2)}-1) &\text{if } 3 \nmid i_2, \,{\rm{lcm}}(i_1,i_2)\le (\ell^2-\ell+1)/{\mu_2},\\
0 &\text{if }3 \nmid i_2,\,  {\rm{lcm}}(i_1,i_2)= \ell^2-\ell+1,\\
3(\ell^2-\ell+1)(\frac{\ell^2-\ell+1}{{\mu_2}{\rm{lcm}}(i_1,i_2/3)}-1) &\text{if }3 \mid i_2. \\
\end{array}
\right.
$$

By \eqref{HurTame} and Proposition \ref{dielleprop} the following result holds.
\begin{theorem} \label{thmB1}
Let ${\mu_2}=\gcd(\ell^2-\ell+1,3)$.
\begin{itemize}
\item For $i_1$ divisor of $(\ell^2-\ell+1)/{\mu_2}$, $i_2$ divisor of $\ell^2-\ell+1$ such that $3\nmid i_2$ and  ${\rm{lcm}}(i_1,i_2)\le (\ell^2-\ell+1)/{\mu_2}$, there exists a subgroup $L$ of $\aut(\cX)$ with
$$
g_L=\frac{1}{2}\left( (\ell+2){\mu_2} i_1i_2-(\ell+1){\mu_2} i_1-\frac{3i_1i_2}{{\rm{lcm}}(i_1,i_2)}\right)+1.
$$

\item For $i_1$ divisor of $(\ell^2-\ell+1)/{\mu_2}$, $i_2$ divisor of $\ell^2-\ell+1$ such that $3\nmid i_2$ and  ${\rm{lcm}}(i_1,i_2)=(\ell^2-\ell+1)$, there exists a subgroup $L$ of $\aut(\cX)$ with
$$
g_L=\frac{1}{2}\left( (\ell+2){\mu_2} i_1i_2-(\ell+1){\mu_2} i_1-\frac{3{\mu_2} i_1 i_2}{\ell^2-\ell+1}\right)+1.
$$

\item For $i_1$ divisor of $(\ell^2-\ell+1)/{\mu_2}$, $i_2$ divisor of $\ell^2-\ell+1$ such that $3\mid i_2$, there exists a subgroup $L$ of $\aut(\cX)$ with
$$
g_L=\frac{1}{2}\left( (\ell+2){\mu_2} i_1i_2-(\ell+1){\mu_2} i_1-\frac{3i_1i_2}{{\rm{lcm}}(i_1,i_2/3)}\right)+1.
$$

\end{itemize}
\end{theorem}

\subsection{$L=(\Sigma_1\rtimes \epsilon(\Upsilon_1)) \times \Sigma_2)$, $\Sigma_1< \epsilon(\Theta)$, $\Sigma_2<\Lambda$}
Here we assume that $p\neq 3$.
Let $\Sigma_1=<\epsilon(T)^{i_1}>$, with $i_1= (\ell^2-\ell+1)/d{\mu_2}$. 
Let $\Sigma_2$ be the group generated by the projectivity defined by the diagonal matrix $[1,1,\beta^{i_2},1]$, with $\beta$ a primitive $(\ell^2-\ell+1)$-th root of unity and $i_2$ a divisor of $\ell^2-\ell+1$;
let $d_2=(\ell^2-\ell+1)/i_2$.

The action of $\pi(L)$ on $\cH$ is described in \cite{cossidente-korchmaros-torres2000}. 
\begin{lemma}[Proposition 4.2 in \cite{cossidente-korchmaros-torres2000}] If ${\mu_2}=1$, then the group $\pi(L)$ has $3$ short orbits, namely  $\{\bar P_1,\bar P_2,\bar P_3\}$ and $2$ short orbits of size $d$ consisting of $\fls$-rational points of $\cH$. 
If ${\mu_2}=3$, then the only short orbit of $\pi(L)$ is $\{\bar P_1,\bar P_2,\bar P_3\}$.
\end{lemma}

As a consequence, the last term $|L_\Lambda|n_2$ of $e_L$ as in Proposition \ref{dielleprop} is just
$$
|\Sigma_2|\sum_{{\bar h}\in {\bar \Sigma_1}, {\bar h}\neq id}\,\,\,\sum_{{\bar P}\in \{\bar P_1,\bar P_2,\bar P_3\},{\bar h}({\bar P})={\bar P}}\,\,\,l_{{\bar h},{\bar P}},
$$
which has already been computed in Section \ref{jji}.
We will distinguish the cases $\mu_2=1$ and $\mu_2=3$.

\subsubsection{${\mu_2}=1$}
Apart from $\{\bar P_1,\bar P_2,\bar P_3\}$, the group $\pi(L)$ has further $2$ short orbits of size $d$ consisting of $\fls$-rational points of $\cH$. 
By \eqref{HurTame} and Proposition \ref{dielleprop} the following result holds.
\begin{theorem}\label{thmB2} Let $p\neq 3$, $\gcd(\ell^2-\ell+1,3)=1$.
For $i_1,i_2$ divisors of $(\ell^2-\ell+1)$  there exists a subgroup $L$ of $\aut(\cX)$ with
$$
g_L=\frac{1}{3}\left(\frac{1}{2}\left( (\ell+2) i_1i_2-(\ell+1) i_1-\frac{3i_1i_2}{{\rm{lcm}}(i_1,i_2)}\right)+1\right).
$$
\end{theorem}

\subsubsection{${\mu_2}=3$}
The only short orbit of $\pi(L)$ is $\{\bar P_1,\bar P_2,\bar P_3\}$.
By \eqref{HurTame} and Proposition \ref{dielleprop} the following result holds.
\begin{theorem}\label{thmB3} Let $\gcd(\ell^2-\ell+1,3)=3$.
\begin{itemize}
\item For $i_1$ divisor of $(\ell^2-\ell+1)/3$, $i_2$ divisor of $\ell^2-\ell+1$ such that $3\nmid i_2$ and  ${\rm{lcm}}(i_1,i_2)\le (\ell^2-\ell+1)/3$, there exists a subgroup $L$ of $\aut(\cX)$ with
$$
g_L=\frac{1}{2}\left( (\ell+2) i_1i_2-(\ell+1) i_1-\frac{i_1i_2}{{\rm{lcm}}(i_1,i_2)}\right)+1.
$$

\item For $i_1$ divisor of $(\ell^2-\ell+1)/3$, $i_2$ divisor of $\ell^2-\ell+1$ such that $3\nmid i_2$ and  ${\rm{lcm}}(i_1,i_2)=(\ell^2-\ell+1)$, there exists a subgroup $L$ of $\aut(\cX)$ with
$$
g_L=\frac{1}{2}\left( (\ell+2) i_1i_2-(\ell+1) i_1-\frac{3 i_1 i_2}{\ell^2-\ell+1}\right)+1.
$$

\item For $i_1$ divisor of $(\ell^2-\ell+1)/3$, $i_2$ divisor of $\ell^2-\ell+1$ such that $3\mid i_2$, there exists a subgroup $L$ of $\aut(\cX)$ with
$$
g_L=\frac{1}{2}\left( (\ell+2) i_1i_2-(\ell+1) i_1-\frac{i_1i_2}{{\rm{lcm}}(i_1,i_2/3)}\right)+1.
$$

\end{itemize}
\end{theorem}

\begin{remark}
{\rm{When $L=\pi^{-1}(\bar G)$ with $\bar G$ a group of type (B), then by Corollary \ref{diellecor2} the genus  $g_L$ coincides with  $g_{\bar L}$. All the possibilities for $g_{\bar L}$ are determined in \cite{cossidente-korchmaros-torres2000}. It should be noted that the statement of Proposition 4.2(3) in \cite{cossidente-korchmaros-torres2000} contains a misprint, as $(q^2-q+1-3n)/6n$ should read $(q^2-q+1+3n)/6n$.}}
\end{remark}

\section{Curves $\cX/L$ with ${\bar L}$ subgroup of a group of type (C)}
In this section we investigate subgroups $L$ of $\aut(\cX)$ with $L=\Sigma_1\times \Sigma_2$, where $\Sigma_1$ is contained  $\Gamma$, $\Sigma_2$ is a subgroup of $\Lambda$, and ${\bar L}$ is a subgroup of a group of type (C). 
To this end, we determine a subgroup $\Omega$ of $\Gamma$ such that $\pi(\Omega)$ is contained in a group of type (C) and $\Omega\cap \Lambda=\{id\}$.
The construction of $\Omega$ requires some technical preliminaries. In particular, a group conjugate to $SU(3,\ell)$ in $GL(3,q^2)$ needs to be considered.

Let $b,c\in \fls$ be such that $b^{\ell+1}=c^\ell+c=-1$, and let
$$
A_2=\begin{pmatrix}
0 & 0 & \frac{1}{b}\\
\frac{c+1}{b} & \frac{-c}{b^2} & 0\\
\frac{-1}{b} & \frac{1}{b^2} & 0
\end{pmatrix}.
$$
Then $A_2$ is a matrix in $GL(3,\ell^2)$ such that $A_2^tD\sigma(A_2)=I_3$. A matrix $M$ is in $SU(3,\ell)$ if and only if $M_2=A_2^{-1}MA_2$ is such that 
\begin{equation}\label{cond2}
M_2^t\sigma(M_2)=I_3,\quad \det(M_2)=1.
\end{equation}
Let ${\mu_1}$ be the largest power of $3$ dividing  $\ell+1$, and let
 $\Pi_{\frac{\ell+1}{{\mu_1}}}$ be the group of $(\frac{\ell+1}{{\mu_1}})$-th roots of unity. As $\gcd(\frac{\ell+1}{{\mu_1}},3)=1$, for each $\lambda \in \Pi_{\frac{\ell+1}{{\mu_1}}}$ there exists a unique $\tilde{\lambda}\in \Pi_{\frac{\ell+1}{{\mu_1}}}$ with $\tilde{\lambda}^{3}=\lambda$.
Let $\Omega_2$ be the the group generated by the matrices
$$
T_1=
\begin{pmatrix}
\frac{w}{\tilde{w}} & 0 & 0 \\
0 & \frac{1}{\tilde{w}} & 0 \\
0 & 0 & \frac{1}{\tilde{w}}
\end{pmatrix},
T_2=
\begin{pmatrix}
\frac{1}{\tilde{w}} & 0 & 0 \\
0 & \frac{w}{\tilde{w}} & 0 \\
0 & 0 & \frac{1}{\tilde{w}}
\end{pmatrix},
U_1=
\begin{pmatrix}
0 & 1 & 0 \\
0 & 0 & 1 \\
1 & 0 & 0
\end{pmatrix},
U_2=
\begin{pmatrix}
0 & 0 & -1 \\
0 & -1 & 0 \\
-1 & 0 & 0
\end{pmatrix},
$$
where $w$ is a primitive $(\frac{\ell+1}{\mu_1})$-th root of unity.
Let $\Theta_1=<T_1>$, $\Theta_2=<T_2>$, $\Upsilon_1=<U_1>$, $\Upsilon_2=<U_2>$, and $\Upsilon=<U_1,U_2>$.

It is straightforward to check that every matrix in $\Omega_2$ satisfies \eqref{cond2}, and that $\Omega_2=(\Theta_1\times \Theta_2)\rtimes \Upsilon$. Moreover, $\Upsilon$ is isomorphic to $Sym_3$. For a matrix $M_2\in \Omega_2 $ let $\epsilon(M_2)$ be the projectivity in $PGL(4,\ell^2)$ defined by the $4\times 4$ matrix obtained from $A_2M_2A_2^{-1}$ by adding $0,0,1,0$ as a third row and as a third column. Then $\epsilon:\Omega_2\to \Gamma$ is an injective group homomorphism. Let $\Omega=\epsilon(\Omega_2)$. 

Let ${\bar P}_i$ be the point in $PG(2,\ell^2)$ whose homogeneous coordinates are the elements in the $i$-th column of $A_2$. Then the subgroup $\pi(\Omega)$ is contained in the stabilizer of the triangle ${\bar P}_1{\bar P}_2{\bar P}_3$ in $PGU(3,\ell)$.   

We now prove that $\Omega\cap \Lambda=\{id\}$. Let $\alpha\in \Omega \cap \Lambda$.  Since $\pi(\alpha)$ fixes pointwise the set of points in $\cH$ belonging to the triangle ${\bar P}_1{\bar P}_2{\bar P}_3$ in $PGU(3,\ell)$, $\alpha \in \epsilon(\Theta_1\times \Theta_2)$. Taking into account that every non-trivial element in $\Gamma\cap \Lambda$ has order $3$, and that $3$ does not divide the order of $\Theta_1\times \Theta_2$ by construction, we obtain that $\alpha=1$. 

Therefore, the following lemma holds.
\begin{lemma}
The subgroup generated by $\Omega$ and $\Lambda$ is the direct product $\Omega\times \Lambda$. The projection $\pi(\Omega\times \Lambda)$ is a subgroup of a group of type {\rm{(C)}} isomorphic to $\Omega_2$.
\end{lemma}

Note that the action of $\pi(\Omega)$ on $\cH$ can be viewed as the action of the group of projectivities defined by the matrices in $\Omega_2$ on the set of points of the plane curve with equation $X^{\ell+1}+Y^{\ell+1}+T^{\ell+1}=0$.

For a divisor $d$ of $(\ell+1)/{\mu_1}$ and for $i=1,2$, let $C_{d}^{(i)}$ be the subgroup of $\epsilon(\Theta_i)$ of order $d$.
We consider subgroups $\Sigma_1$ of $\Omega$ of the following types:
\begin{itemize}
\item[(a)] $C^{(1)}_{d_1}\times C^{(2)}_{d_2}$;

\item[(b)] the cyclic subgroup of order $(\ell+1)/({\mu_1} d_1)$ generated by $\epsilon(T_1)^{d_1}\epsilon(T_2)^{2d_1}$, with $d_1$ a divisor of $(\ell+1)/{\mu_1}$;

\item[(c)] $(C^{(1)}_{d_1}\times C^{(2)}_{d_1})\rtimes \Upsilon_2$;
\item[(d)]  $(C^{(1)}_{d_1}\times C^{(2)}_{d_1})\rtimes \Upsilon_1$;
\item[(e)]  $(C^{(1)}_{d_1}\times C^{(2)}_{d_1})\rtimes \Upsilon$.
\end{itemize}
Cases (a)-(e) are dealt with separately.

\subsection{$L=\Sigma_1\times \Sigma_2$, with $\Sigma_1$ as in (a), $\Sigma_2<\Lambda$, $|\Sigma_2|=d$}
The action of $\pi(\Sigma_1)$ on $\cH$ was investigated in \cite[Example 5.11]{garcia-stichtenoth-xing2001}.
Any non-trivial element in $\pi(\Sigma_1)$ fixes no point in $\bar {\mathcal O}_2$. Moreover,
$$
\sum_{{\bar h} \in \pi(\Sigma_1), {\bar h}\neq id}|\{\bar P \in \bar {\mathcal O}_1\mid \bar h(\bar P)=\bar P\}|=(\ell+1)(d_1+d_2+\gcd(d_1,d_2)-3).
$$
By \eqref{HurTame} and Proposition \ref{dielleprop} the following result holds.
\begin{theorem}\label{thmC1} Let ${\mu_1}$ be the highest power of $3$ which divides $\ell+1$. For any two divisors $d_1,d_2$ of $(\ell+1)/{\mu_1}$, and for any $d$ divisor of $\ell^2-\ell+1$, there exists a subgroup $L$ of $\aut(\cX)$ with
$$
g_L=1+\frac{(\l^3+1)(\l^2-d-1)-d(\l+1)(d_1+d_2+\gcd(d_1,d_2)-3)}{2dd_1d_2}.
$$
\end{theorem}
\subsection{$L=\Sigma_1\times \Sigma_2$, with $\Sigma_1$ as in (b), $\Sigma_2<\Lambda$, $|\Sigma_2|=d$}
By \cite[Example 5.10]{garcia-stichtenoth-xing2001}
any non-trivial element in $\pi(\Sigma_1)$ fixes no point in $\bar {\mathcal O}_2$. Moreover,
$$
g_{\bar L}=1+\frac{\mu_1 d_1(\l-2)(\l+1)}{2(\l+1)}
$$
when $(\ell+1)/({\mu_1} d_1)$ is odd, whereas
$$
g_{\bar L}=1+\frac{\mu_1 d_1(\l-3)(\l+1)}{2(\l+1)}
$$
when $(\ell+1)/({\mu_1} d_1)$ is even.
By  Proposition \ref{dielleprop2} the following result holds.
\begin{theorem}\label{thmC2} Let ${\mu_1}$ be the highest power of $3$ which divides $\ell+1$. 
\begin{itemize}
\item For any divisor $d_1$ of $(\ell+1)/{\mu_1}$ such that $(\ell+1)/({\mu_1} d_1)$ is odd, and for every divisor $d$ of $\ell^2-\ell+1$, there exists a subgroup $L$ of $\aut(\cX)$ with
$$
g_L=\frac{1}{2}\left( {\mu_1} d_1\left(\frac{\ell^2-\ell+1}{d}\right)(\ell^2-d-1)+2\right).
$$
\item For any divisor $d_1$ of $(\ell+1)/{\mu_1}$ such that $(\ell+1)/({\mu_1} d_1)$ is even, and for every divisor $d$ of $\ell^2-\ell+1$, there exists a subgroup $L$ of $\aut(\cX)$ with
$$
g_L=\frac{1}{2}\left( {\mu_1} d_1\left(\frac{\ell^2-\ell+1}{d}\right)(\ell^2-d-1)-{\mu_1} d_1+2\right).
$$
\end{itemize}

\end{theorem}

\subsection{$p\neq 2$, $L=\Sigma_1\times \Sigma_2$, with $\Sigma_1$ as in (c), $\Sigma_2<\Lambda$, $|\Sigma_2|=d$}\label{schleck}
As $C_{d_1}^{(1)}\times C_{d_1}^{(2)}$ is a normal subgroup of $\Sigma_1$,   the subgroup $\pi(\Upsilon_2)$ acts on the quotient curve $\cH/\pi(C_{d_1}^{(1)}\times C_{d_1}^{(1)})$.
Such action is equivalent to the action of the involutory projectivity defined by $U_2$ on the plane curve $\cE$ with equation $X^{\frac{\ell+1}{ d_1}}+Y^{\frac{\ell+1}{d_1}}+T^{\frac{\ell+1}{d_1}}=0$. The fixed points of $U_2$ on $\cE$ are the points on the line $X=T$, together with $(-1,0,1)$ if $(\ell+1)/d_1$ is odd. It is straightforward to check that any point of $\cH$ lying over one of those fixed points of $\cX$ is $\fl$-rational.
As $\cE$ is non-singular, the genus ${\bar g}$ of $\cE/U_2$ is given by
$$
\left(\frac{\ell+1}{d_1}-1\right)\left(\frac{\ell+1}{d_1}-2\right)-2=2(2{\bar{g}}-2)+2\lceil(\ell+1)/(2d_1)\rceil,
$$
that is
\begin{equation}\label{contador}
{\bar g}=\frac{((\ell+1)/(d_1)-2)^2}{4},\quad \text{ if } (\ell+1)/ d_1\text{ is even,}
\end{equation}
and
\begin{equation}\label{evans}
{\bar g}=\frac{((\ell+1)/(d_1)-3)((\ell+1)/(d_1)-1)}{4},\quad  \text{ if } (\ell+1)/d_1 \text{ is odd.}
\end{equation}
Note that no point in $\bar {\mathcal O}_2$ is fixed by a non trivial element in $\pi(\Sigma_1)$. Also,
since $\cE/U_2$ is isomorphic to $\cH/\pi(\Sigma_1)$, from Hurwitz's genus formula it follows
$$
\sum_{{\bar h} \in \pi(\Sigma_1), {\bar h}\neq id}|\{\bar P \in \bar {\mathcal O}_1\mid \bar h(\bar P)=\bar P\}|=\ell^2-\ell-2-2d_1^2(2\bar g-2).
$$

By \eqref{HurTame} and Proposition \ref{dielleprop2} the following result holds.
\begin{theorem}\label{thmC3} Assume that $p\neq 2$. Let ${\mu_1}$ be the highest power of $3$ dividing $\ell+1$. 
 For any divisor $d_1$ of $(\ell+1)/{\mu_1}$, and for every divisor $d$ of $\ell^2-\ell+1$, there exists a subgroup $L$ of $\aut(\cX)$ with
$$
g_L=\bar g+\frac{(\ell+1)(\ell^2-1)}{4d_1^2}\left(\frac{\ell^2-\ell+1}{d}-1\right),
$$
with $\bar g$ as in \eqref{contador} for $(\ell+1)/d_1$ even, and with $\bar g$ as in \eqref{evans} for  $(\ell+1)/d_1$ odd.
\end{theorem}

\subsection{$p\neq 3$, $L=\Sigma_1\times \Sigma_2$, with $\Sigma_1$ as in (d), $\Sigma_2<\Lambda$, $|\Sigma_2|=d$}\label{schleck2}
As $C_{d_1}^{(1)}\times C_{d_1}^{(2)}$ is a normal subgroup of $\Sigma_1$,   the subgroup $\pi(\Upsilon_1)$ acts on the quotient curve $\cH/\pi(C_{d_1}^{(1)}\times C_{d_1}^{(2)})$.
Such action is equivalent to the action of the projectivity defined by $U_1$ on the plane curve $\cE$ with equation $X^{\frac{\ell+1}{ d_1}}+Y^{\frac{\ell+1}{d_1}}+T^{\frac{\ell+1}{d_1}}=0$. Let $f$ be a primitive third root of unity in $\fls$. If 
$3$ does not divide $(\ell+1)/d_1$, then the fixed points of $U_1$ on $\cE$ are precisely $(f,f^2,1)$ and $(f^2,f,1)$; otherwise no point on $\cE$ is fixed by $U_1$.
Arguing as in Section \ref{schleck} we get that
$$
\sum_{{\bar h} \in \pi(\Sigma_1), {\bar h}\neq id}|\{\bar P \in \bar {\mathcal O}_1\mid \bar h(\bar P)=\bar P\}|=\ell^2-\ell-2-3d_1^2(2\bar g-2),
$$
where 
\begin{equation}\label{contador2}
{\bar g}=\frac{((\ell+1)/(d_1)-1)((\ell+1)/(d_1)-2)}{6},\quad \text{ if }3 \nmid (\ell+1)/ d_1,
\end{equation}
and
\begin{equation}\label{evans2}
{\bar g}=\frac{((\ell+1)/(d_1)-1)((\ell+1)/(d_1)-2)+4}{6},\quad  \text{ if } 3 \mid (\ell+1)/d_1.
\end{equation}

By \eqref{HurTame} and Proposition \ref{dielleprop2} the following result holds.
\begin{theorem}\label{thmC4} Assume that $p\neq 3$. Let ${\mu_1}$ be the highest power of $3$ dividing $\ell+1$. 
 For any divisor $d_1$ of $(\ell+1)/{\mu_1}$, and for every divisor $d$ of $\ell^2-\ell+1$, there exists a subgroup $L$ of $\aut(\cX)$ with
$$
g_L=\bar g+\frac{(\ell+1)(\ell^2-1)}{6d_1^2}\left(\frac{\ell^2-\ell+1}{d}-1\right),
$$
with $\bar g$ as in \eqref{contador2} when $3$ does not divide $(\ell+1)/d_1$, and with $\bar g$ as in \eqref{evans2} if  $3\mid (\ell+1)/d_1$.
\end{theorem}

\subsection{$p\neq 2,3$, $L=\Sigma_1\times \Sigma_2$, with $\Sigma_1$ as in (e), $\Sigma_2<\Lambda$, $|\Sigma_2|=d$}
As $C_{d_1}^{(1)}\times C_{d_1}^{(2)}$ is a normal subgroup of $\Sigma_1$,   the subgroup $\pi(\Upsilon)$ acts on the quotient curve $\cH/\pi(C_{d_1}^{(1)}\times C_{d_1}^{(2)})$. Such action is equivalent to that of all permutations of the coordinates $(X,Y,T)$ on the plane curve $\cE$ with equation $X^{\frac{\ell+1}{ d_1}}+Y^{\frac{\ell+1}{d_1}}+T^{\frac{\ell+1}{d_1}}=0$. It is straightforward to check that:
\begin{itemize} 
\item $(X,Y,T)\mapsto (T,Y,X)$ fixes the points on the line $X=T$, together with $(-1,0,1)$ if $(\ell+1)/d_1$ is odd;
\item $(X,Y,T)\mapsto (X,T,Y)$ fixes the points on the line $Y=T$, together with $(0,-1,1)$ if $(\ell+1)/d_1$ is odd;
\item $(X,Y,T)\mapsto (Y,X,T)$ fixes the points on the line $X=Y$, together with $(-1,1,0)$ if $(\ell+1)/d_1$ is odd;
\item $(X,Y,T)\mapsto (Y,T,X)$ fixes the points  $(f,f^2,1)$ and $(f^2,f,1)$ if $3$ does not divide $(\ell+1)/d_1$, and fixes no point if $3\mid (\ell+1)/d_1$;
\item $(X,Y,T)\mapsto (T,X,Y)$ fixes the points  $(f,f^2,1)$ and $(f^2,f,1)$ if $3$ does not divide $(\ell+1)/d_1$, and fixes no point if $3\mid (\ell+1)/d_1$.
\end{itemize}
Therefore,  for the genus $\bar g$ of the quotient curve of $\cH/\pi(C_{d_1}^{(1)}\times C_{d_1}^{(1)})$ with respect to $\pi(\Upsilon)$ it turns out that
\begin{equation}\label{bugno}
\bar g=\frac{1}{12}\left(\left(\frac{\ell+1}{d_1}\right)^2-6\frac{\ell+1}{d_1}+o\right),
\end{equation}
where
$$
o=\begin{cases} 
12 \text{ if } 3\mid (\ell+1)/d_1,2\mid (\ell+1)/d_1,  \\ 
9 \text{ if } 3\mid (\ell+1)/d_1, 2 \nmid (\ell+1)/d_1,\\ 
8 \text{ if } 3 \nmid (\ell+1)/d_1, 2\mid (\ell+1)/d_1,\\ 
5 \text{ if } 3\nmid (\ell+1)/d_1, 2 \nmid (\ell+1)/d_1
\end{cases}.
$$
Arguing as in the preceedings Sections, we obtain the following result as a consequence of  \eqref{HurTame} and Proposition \ref{dielleprop2}.
\begin{theorem}\label{thmC5} Assume that $p\neq 3$. Let ${\mu_1}$ be the highest power of $3$ dividing $\ell+1$. 
 For any divisor $d_1$ of $(\ell+1)/{\mu_1}$, and for every divisor $d$ of $\ell^2-\ell+1$, there exists a subgroup $L$ of $\aut(\cX)$ with
$$
g_L=\bar g+\frac{(\ell+1)(\ell^2-1)}{12d_1^2}\left(\frac{\ell^2-\ell+1}{d}-1\right),
$$
with $\bar g$ as in \eqref{bugno}.
\end{theorem}

\begin{remark}
{\rm{When $L=\pi^{-1}(\bar G)$ with $\bar G$ a group of type (C), then by Corollary \ref{diellecor2} the genus  $g_L$ coincides with  $g_{\bar L}$. Some possibilities for $g_{\bar L}$ are determined in \cite{garcia-stichtenoth-xing2001} when either $\bar G$ is isomorphic to $\Sigma_1$ as in cases (a)-(b), or $\bar G$ is isomorphic to $\Upsilon$.
It should be noted that other possibilities for $g_{\bar L}$ are computed here, namely the integers $\bar g$ as in \eqref{contador},\eqref{evans},\eqref{contador2},\eqref{evans2},\eqref{bugno}. To our knowledge these integers provide genera of quotient curves of the Hermitian curve that have not appeared in the literature so far.
}}
\end{remark}

\section{Curves $\cX/L$ with ${\bar L}$ subgroup of a group of type (D)}
In this section we will consider the case where ${\bar L}$ is contained in one of the following subgroups $\bar G_1$ and $\bar G_2$ of $PGU(3,\ell)$ stabilizing the line with equation $X=0$.
\begin{itemize}
\item Let $\Psi_1$ be the subgroup of $GL(3,\ell^2)$ consisting of matrices
$$
M_{a_{1},a_2,a_3,a_4}=
\begin{pmatrix} 1 & 0 & 0 \\
0 & a_1 & a_2 \\
0 & a_3 & a_4
\end{pmatrix}
$$
with $a_1^\ell=a_1$, $a_4^\ell=a_4$, $a_3^\ell=-a_3$, $a_2^\ell=-a_2$, $a_1a_4-a_2a_3=1$. Let $\bar G_1$ be the subgroup of $PGL(3,\ell^2)$ of the projectivities defined by the matrices in $\Psi_1$. Clearly $\bar G_1$ is isomorphic to $\Psi_1$. It is straightforward to check that $\Psi_1$ is a subgroup of $SU(3,\ell)$; in particular,
 $\bar G_1$ is a subgroup of $PGU(3,\ell)$ preserving the line $X=0$. Also, it is easily seen that the map
$$
M_{a_{1},a_2,a_3,a_4}\mapsto \begin{pmatrix}a_1 & \lambda a_2 \\ {a_3}{\lambda^{-1}} & a_4\end{pmatrix}
$$
with $\lambda^{\ell-1}=-1$,
defines a isomorphism of $\Psi_1$ and $SL(2,\ell)$, and the action of $\bar G_1$ on the points of $\cH$ on the line $X=0$ is isomorphic to the action of $SL(2,\ell)$ on the projective line $PG(1,\ell)$.
\item Let $\bar G_2$ be the subgroup of $GL(3,\ell^2)$ generated by the projectivities defined by the matrices
$$
W=\begin{pmatrix}1 & 0 & 0 \\ 0 & 0 & 1 \\0 & 1 & 0\end{pmatrix}, \quad M_a=\begin{pmatrix}a & 0 & 0 \\ 0 & a^{\ell+1} & 0 \\ 0 & 0 & 1\end{pmatrix}
$$
with $a$ ranging over the set of non-zero elements in $\fls$. It is easily seen that $\bar G_2$ is a subgroup of $PGU(3,\ell)$ preserving the line with equation $X=0$, and that $\bar G_2$ is isomorphic to the diedral group of order $2(\ell^2-1)$.

\end{itemize}

\subsection{${\bar L}$ subgroup of $\bar G_1$}
Let $\Omega$ be the subgroup of $PGL(4,\ell^2)$ consisting of matrices ${\tilde M}_{a_1,a_2,a_3,a_4}$, with ${M}_{a_1,a_2,a_3,a_4}\in \Psi_1$. As $\Psi_1$ is contained in $SU(3,\ell)$, we have that $\Omega$ is actually a subgroup of $\Gamma$. It is straightforward to check that $\Omega\cap \Lambda$ is trivial.
Also, $\bar G_1=\pi(\Omega)$ clearly holds.

\begin{lemma} Any non-trivial element in $\bar G_1$ fixes no point of $\cH$ outside the line with equation $X=0$.
\end{lemma}
\begin{proof} Assume that ${M}_{a_1,a_2,a_3,a_4}(x_1,y_1,t_1)^t=\varrho (x_1,y_1,t_1)^t$ for some $\varrho \in {\mathbb K}$,  $\varrho\neq 0$, and for some   $(x_1,y_1,t_1)\in {\mathbb K}^3$ with $x_1\neq 0$ and $x_1^{\ell+1}=y_1^\ell t_1+ y_1 t_1^\ell$. Clearly this can only happen for $\varrho=1$ and $y_1t_1\neq 0$, whence we can assume that $t_1=1$ and that
$$
(a_1-1)y_1+a_2=0,\qquad a_3y_1+(a_4-1)=0.
$$
Taking $\ell$-th powers in both equalities we then have
$$
(a_1-1)y_1^\ell-a_2=0,\qquad -a_3y_1^\ell+(a_4-1)=0.
$$
It is then straightforward to deduce that $a_1=1$, $a_2=0$, $a_3=0$, $a_4=1$, that is ${M}_{a_1,a_2,a_3,a_4}=I_3$.
\end{proof}

Assume that $L$ is a tame subgroup of $\aut(\cX)$ with $L=\Sigma_1\times \Sigma_2$, where $\Sigma_1 <\Omega$ and $\Sigma_2<\Lambda$. Let $d=|\Sigma_2|$. By Proposition \ref{dielleprop2} we then have that
\begin{equation}\label{mosn}
g_L=g_{\bar L}+\frac{(\ell^3+1)(\ell^2-d-1)-d(\ell^2-d-2)}{2d|\Sigma_1|}
\end{equation}
where $g_{\bar L}$ is the genus of the quotient curve $\cH/\pi(\Sigma_1)$.

In \cite{cossidente-korchmaros-torres2000} a number of genera of quotient curves $\cH/\pi(\Sigma_1)$ with $\Sigma_1$ a subgroup of $\Omega$ are computed.
The following results are then  straightforward consequences of \eqref{mosn} together with Proposition 3.3 in \cite{cossidente-korchmaros-torres2000}.
\begin{theorem}\label{7punto1} Assume that $p=2$. 

Let $d$ be any divisor of $\ell^2-\ell+1$. Then there exist subgroups $L$ of $\aut(\cX)$ with the following properties.
\begin{itemize}
\item For any $e\mid \ell+1$, $|L|=de$ and
$$
g_L=\frac{1}{2e}\left(\frac{\ell^3+1}{d}(\ell^2-d-1)\right)+1.
$$


\item For any $e\mid \ell-1$, $|L|=de$ and
$$
g_L=\frac{1}{2e}\left(\frac{\ell^3+1}{d}(\ell^2-d-1)-2(e-1)\right)+1.
$$


\end{itemize}

\end{theorem}

\begin{theorem}\label{7punto2} Assume that $p\neq 2$. Let $d$ be any divisor of $\ell^2-\ell+1$. Then there exist subgroups $L$ of $\aut(\cX)$ with the following properties.
%
%

\begin{itemize}

\item For any divisor $e$ of $(\ell+1)/2$,  $|L|=2de$ and
$$
g_L=\frac{1}{4e}\left(\frac{\ell^3+1}{d}(\ell^2-d-1)-(\ell+1)\right)+1.
$$

\item For any divisor $e$ of $(\ell+1)/2$,
$|L|=4de$ and
$$
g_L=\frac{1}{8e}\left(\frac{\ell^3+1}{d}(\ell^2-d-1)-(\ell+1)-2o\right)+1,
$$
with 
$$
o=\begin{cases}
2e & \text{ if } \ell \equiv 1 \pmod 4 \\
0 & \text{ if } \ell \equiv 3 \pmod 4 \\
\end{cases}.
$$

\item For any divisor $e$ or $(\ell-1)/2$, 
$|L|=2de$ and
$$
g_L=\frac{1}{4e}\left(\frac{\ell^3+1}{d}(\ell^2-d-1)-(\ell+1)-4e+4\right)+1.
$$

\item For any divisor $e$ or $(\ell-1)/2$, $|L|=4de$ and
$$
g_L=\frac{1}{8e}\left(\frac{\ell^3+1}{d}(\ell^2-d-1)-(\ell+1)-2o\right)+1
$$
with 
$$
o=\begin{cases}
4e-2 & \text{ if } \ell \equiv 1 \pmod 4 \\
2e-2 & \text{ if } \ell \equiv 3 \pmod 4 \\
\end{cases}.
$$

\item When $p\ge 5$, $\ell^2\equiv 1\pmod {16}$,  $|L|=48d$ and
$$
g_L=\frac{1}{96}\left(\frac{\ell^3+1}{d}(\ell^2-d-1)-(\ell+1)-2o\right)+1,
$$
with 
$$
o=\begin{cases}
46 & \text{ if } \ell \equiv 1 \pmod 4 \text{ and } \ell \equiv 1 \pmod 3\\
30 & \text{ if } \ell \equiv 1 \pmod 4 \text{ and } \ell \equiv 2 \pmod 3\\
16 & \text{ if } \ell \equiv 3 \pmod 4 \text{ and } \ell \equiv 1 \pmod 3\\
0 & \text{ if } \ell \equiv 3 \pmod 4 \text{ and } \ell \equiv 2 \pmod 3
\end{cases}.
$$

\item When $p\ge 5$,  $|L|=24d$ and
$$
g_L=\frac{1}{48}\left(\frac{\ell^3+1}{d}(\ell^2-d-1)-(\ell+1)-2o\right)+1,
$$
with 
$$
o=\begin{cases}
22 & \text{ if } \ell \equiv 1 \pmod 4 \text{ and } \ell \equiv 1 \pmod 3\\
6 & \text{ if } \ell \equiv 1 \pmod 4 \text{ and } \ell \equiv 2 \pmod 3\\
16 & \text{ if } \ell \equiv 3 \pmod 4 \text{ and } \ell \equiv 1 \pmod 3\\
0 & \text{ if } \ell \equiv 3 \pmod 4 \text{ and } \ell \equiv 2 \pmod 3
\end{cases}.
$$

\item When $p\ge 7$, $\ell^2\equiv 1 \pmod 5$, $|L|=120d$ and
$$
g_L=\frac{1}{240}\left(\frac{\ell^3+1}{d}(\ell^2-d-1)-(\ell+1)-2o\right)+1
$$
with
$$
o=\begin{cases}
118 & \text{ if } \ell \equiv 1\pmod 3, \,\ell \equiv 1 \pmod 4 \text{ and } \ell \equiv 1 \pmod 5\\
78 & \text{ if } \ell \equiv 2\pmod 3,\,\ell \equiv 1 \pmod 4 \text{ and } \ell \equiv 1 \pmod 5\\
78 & \text{ if } \ell \equiv 1\pmod 3,\,\ell \equiv 3 \pmod 4 \text{ and } \ell \equiv 1 \pmod 5\\
48 & \text{ if } \ell \equiv 2\pmod 3,\,\ell \equiv 3 \pmod 4 \text{ and } \ell \equiv 1 \pmod 5\\
70 & \text{ if } \ell \equiv 1\pmod 3,\,\ell \equiv 1 \pmod 4 \text{ and } \ell \equiv 4 \pmod 5\\
30 & \text{ if } \ell \equiv 2\pmod 3,\,\ell \equiv 1 \pmod 4 \text{ and } \ell \equiv 4 \pmod 5\\
40 & \text{ if } \ell \equiv 1\pmod 3,\,\ell \equiv 3 \pmod 4 \text{ and } \ell \equiv 4 \pmod 5\\
0 & \text{ if } \ell \equiv 2\pmod 3,\,\ell \equiv 3 \pmod 4 \text{ and } \ell \equiv 4 \pmod 5
\end{cases}.
$$
\end{itemize}
\end{theorem}

\begin{remark}
\rm{It should be noted that by the discussion in \cite[Sect. 3]{cossidente-korchmaros-torres2000}, Theorems \ref{7punto1} and \ref{7punto2} describe  genera $g_L$ for all tame subgroups $L=\Sigma_1\times \Sigma_2$ such that $\Sigma_1$ is a subgroup of  $\Omega$ containing the diagonal matrix $J=[1,-1,1,-1]$, and $\Sigma_2<\Lambda$.}
\end{remark}

\subsection{${\bar L}$ subgroup of $\bar G_2$}
We determine a subgroup $\Omega$ of $\Gamma$ such that $\pi(\Omega)$ is contained in $\bar G_2$ and $\Omega\cap \Lambda=\{id\}$ holds.
Let ${\mu_1}$ be the maximum power of $3$ dividing $\ell+1$, and let $\gamma$ be a primitive $((\ell^2-1)/{\mu_1})$-th root of unity in $\fls$. Let $\Omega$ the group generated by the projectivities defined by the matrices $V_\gamma=[\gamma^{2-\ell},\gamma^{\ell+1},\gamma, 1]$ and 
$${\bar W}
=\begin{pmatrix}
-1 & 0 & 0 & 0\\
0 & 0 & 0 & 1\\
0 & 0 & 1 & 0\\
0 & 1 & 0 & 0
\end{pmatrix}.
$$
\begin{lemma} The group $\pi(\Omega)$ is contained in $\bar G_2$.
\end{lemma}
\begin{proof} The image by $\pi$ of the projectivity defined by $V_\gamma$ is associated to the matrix
$$
\begin{pmatrix}\gamma^{2-\ell} & 0 & 0 \\ 0 & \gamma^{\ell+1} & 0 \\ 0 & 0 & 1\end{pmatrix}.
$$
This is the matrix $M_a$ for $a=\gamma^{2-\ell}$, as $(\gamma^{2-\ell})^{\ell+1}=\gamma^{-(\ell^2-1-\ell-1)}=\gamma^{\ell+1}$.

The image by $\pi$ of the projectivity defined by $\bar W$ is  associated to the matrix
$$
W'=\begin{pmatrix}-1 & 0 & 0 \\ 0 &  0 & 1 \\ 0 & 1 & 0\end{pmatrix}.
$$
Since for $a=-1$ we have $M_{a}W'=W$, this projectivity belong to $\bar G_2$, and the proof is complete. 
\end{proof}

\begin{lemma} The intersection $\Omega\cap \Lambda$ is trivial.
\end{lemma}
\begin{proof} Any element $\alpha$ in $\Omega\cap \Lambda$  is defined by a diagonal matrix of $\Omega$, that is $[\gamma^{i(2-\ell)},\gamma^{i(\ell+1)},\gamma^i, 1]$ for some $i$.
Moreover, $\gamma^{i(\ell^2-\ell+1)}=1$. Since $\gamma^{i(\ell^2-1)/{\mu_1}}=1$ and $\gcd(\ell^2-\ell+1,(\ell^2-1)/{\mu_1})=1$, this can only happen $\gamma^i=1$. Then the assertion is proved.
\end{proof}
By the above lemmas, the map $\pi$ defines an isomorphism of $\Omega$  onto the subgroup of $\bar G_2$ generated by the projectivities defined by $W$ and $M_a$, with $a$ an $((\ell^2-1)/{\mu_1})$-th root of unity. A number of genera of quotient curves $\cH/\pi(\Sigma_1)$ for $\Sigma_1$ tame subgroup of $\pi(\Omega)$ have been established in \cite[Thm. 5.4, Ex. 5.6]{garcia-stichtenoth-xing2001}. Taking into accout Proposition \ref{dielleprop2}, we are in a position to compute genera $g_L$ for tame subgroups $L=\Sigma_1\times \Sigma_2$ with $\Sigma_1<\Omega$ and $\Sigma_2<\Lambda$.
\begin{theorem}\label{7punto3}
Assume that $p\neq 2$. Let $\mu_1$ be the highest power of $3$ dividing $\ell+1$. Let $d$ be any divisor of $\ell^2-\ell+1$. Then there exist subgroups $L$ of $\aut(\cX)$ with the following properties.
\begin{itemize}
\item For any divisor $e$ of $(\ell^2-1)/{\mu_1}$, $|L|=2ed$ and
$$
g_L=\frac{1}{4e}\left(\frac{\ell^3+1}{d}(\ell^2-d-1)+ (\ell+1)(1-u-\tilde u)+2(e+u)-\delta\right),
$$
where $u=\gcd(e,\ell+1)$, $\tilde u=\gcd(e,\ell-1)$, 
$$\delta
=\begin{cases}0 & \text{ if } e \text{ divides } (\ell^2-1)/2 \\
e & \text{ otherwise }\end{cases}.
$$

\item When $\ell\equiv 1 \pmod 4$, for any even divisor $e$ of $\ell-1$, $|L|=2ed$ and
$$
g_L=\frac{1}{4e}\left(\frac{(\ell^3+1)}{d}(\ell^2-d-1)-\ell+3\right).
$$

\item When $\ell\equiv 3 \pmod 4$, for any even divisor $e$ of $\ell-1$, $|L|=2ed$ and
$$
g_L=\frac{1}{4e}\left(\frac{(\ell^3+1)}{d}(\ell^2-d-1)-\ell+2e+3\right).
$$

\item When $\ell\equiv 1 \pmod 4$, for any odd divisor $e$ of $\ell-1$, $|L|=2ed$ and
$$
g_L=\frac{1}{4e}\left(\frac{(\ell^3+1)}{d}(\ell^2-d-1)+2\right).
$$

\item When $\ell\equiv 3 \pmod 4$, for any odd divisor $e$ of $\ell-1$, $|L|=2ed$ and
$$
g_L=\frac{1}{4e}\left(\frac{(\ell^3+1)}{d}(\ell^2-d-1)+2e+2\right).
$$

\end{itemize}

\end{theorem}

\begin{remark} \rm{When $L=\pi^{-1}(\bar G)$ with $\bar G$ a group of type (D), then by Corollary \ref{diellecor2} the genus  $g_L$ coincides with  $g_{\bar L}$. As  already mentioned in the present section, a number of possibilities for $g_{\bar L}$ are determined in \cite[Prop. 3.3]{cossidente-korchmaros-torres2000}, and in \cite[Thm. 5.4]{garcia-stichtenoth-xing2001}.}
\end{remark}


\section{The case $\ell=5$}\label{elle5}
To exemplify how the results of this paper provide many new genera for $\fqs$-maximal curves we consider in this section the case $q=5^3$.
Up to our knowledge, the $32$ integers listed in the following table are new values in the spectrum of genera of ${\mathbb F}_{5^6}$-maximal curves.

\begin{center}
New genera of ${\mathbb F}_{5^6}$-maximal curves
 \begin{tabular}{|c|c || c|c |}\hline
$g$ & Ref. & $g$ & Ref.  \\\hline
5 & \ref{7punto3} & 
9 & \ref{thmA1}(i), \ref{7punto3} \\
14& \ref{thmA1}(ii)  &
21& \ref{thmA1}(ii) \\
22& \ref{thmB1}, \ref{thmB3} &
25& \ref{thmC5} \\
27& \ref{7punto3} &
37& \ref{thmC1}, \ref{thmC4}, \ref{7punto2}, \ref{7punto3} \\
38& \ref{thmA1}(i), \ref{thmA1}(ii), \ref{7punto2}, \ref{7punto3} &
70& \ref{thmA1}(ii), \ref{thmB3} \\
73& \ref{thmC5} &
74& \ref{7punto2} \\
76& \ref{thmA1}(i), \ref{thmA1}(ii), \ref{thmC1}, \ref{thmC2}, \ref{thmC3}, \ref{7punto2}, \ref{7punto3}&
77& \ref{7punto3} \\
86& \ref{thmA1}(iv) &
109& \ref{thmC1},  \ref{7punto3}\\
121& \ref{thmC4},  \ref{7punto2} &
140& \ref{thmA1}(ii) \\
148& \ref{thmB3},  \ref{thmC4} &
180& \ref{thmC3},  \ref{7punto3}\\
208& \ref{thmB1} &
220&  \ref{thmA1}(i),  \ref{thmA1}(ii),  \ref{thmC1}, \ref{thmC2}, \ref{thmC3},  \ref{7punto2},  \ref{7punto3}\\
221& \ref{7punto3} &
241&  \ref{thmC5} \\
242& \ref{7punto2} &
282& \ref{thmA1}(iv)\\
361& \ref{thmC1},  \ref{7punto3}&
362&  \ref{thmA1}(i),  \ref{thmA1}(ii),  \ref{7punto2},  \ref{7punto3}\\
442& \ref{thmA1}(iv),  \ref{thmB1},  \ref{thmC1}, \ref{thmC2}&
484& \ref{thmB3},  \ref{thmC4}\\
724&  \ref{thmA1}(i),  \ref{thmA1}(ii),   \ref{thmC1}, \ref{thmC2}, \ref{thmC3}, \ref{7punto2}, \ref{7punto3}&
725& \ref{7punto3}\\ \hline
\end{tabular}
\end{center}

\end{document}